\documentclass[11pt]{article}

\usepackage[letterpaper, margin=1in]{geometry}
\usepackage[utf8]{inputenc}
\usepackage[T1]{fontenc}
\usepackage[hypertexnames=false]{hyperref}
\usepackage{color}
\hypersetup{colorlinks=true,citecolor=blue, linkcolor=blue, urlcolor=blue}
\usepackage{url}
\usepackage{booktabs}
\usepackage{amsfonts}
\usepackage{nicefrac}
\usepackage{amsmath, amsthm, amssymb}
\usepackage{microtype}
\ifxetex 
  \usepackage[libertine]{newtxmath}
\else
  \usepackage{newtxmath}
\fi
\usepackage[tt=false]{libertine} 
\usepackage{xcolor}

\usepackage{float}
\usepackage{graphicx}
\usepackage{makecell}
\usepackage{enumitem}
\usepackage{tikz}
\usepackage{aliascnt}
\usetikzlibrary{calc}

\newtheorem{theorem}{Theorem}
\newaliascnt{observation}{theorem}

\aliascntresetthe{observation}
\newaliascnt{claim}{theorem}

\aliascntresetthe{claim}
\newtheorem*{claim*}{Claim}
\newaliascnt{condition}{theorem}

\aliascntresetthe{condition}
\newaliascnt{example}{theorem}

\aliascntresetthe{example}
\newaliascnt{fact}{theorem}

\aliascntresetthe{fact}
\newaliascnt{lemma}{theorem}
\newtheorem{lemma}[lemma]{Lemma}
\aliascntresetthe{lemma}
\newaliascnt{proposition}{theorem}

\aliascntresetthe{proposition}
\newaliascnt{corollary}{theorem}
\newtheorem{corollary}[corollary]{Corollary}
\aliascntresetthe{corollary}
\theoremstyle{definition}

\newaliascnt{definition}{theorem}
\newtheorem{definition}[definition]{Definition}
\aliascntresetthe{definition}
\newaliascnt{remark}{theorem}

\aliascntresetthe{remark}
\newtheorem*{remark*}{Remark}
\newtheorem{assumption}{Assumption}

\usepackage{cleveref}
\crefname{theorem}{Theorem}{Theorems}
\Crefname{theorem}{Theorem}{Theorems}
\crefname{lemma}{Lemma}{Lemmas}
\Crefname{lemma}{Lemma}{Lemmas}
\crefname{corollary}{Corollary}{Corollaries}
\Crefname{corollary}{Corollary}{Corollaries}
\crefname{definition}{Definition}{Definitions}
\Crefname{definition}{Definition}{Definitions}
\crefname{proposition}{Proposition}{Propositions}
\Crefname{proposition}{Proposition}{Propositions}
\crefname{fact}{Fact}{Facts}
\Crefname{fact}{Fact}{Facts}
\crefname{remark}{Remark}{Remarks}
\Crefname{remark}{Remark}{Remarks}
\crefname{condition}{Condition}{Conditions}
\Crefname{condition}{Condition}{Conditions}
\crefname{observation}{Observation}{Observations}
\Crefname{observation}{Observation}{Observations}
\crefname{assumption}{Assumption}{Assumptions}
\Crefname{assumption}{Assumption}{Assumptions}
\crefname{appendix}{Appendix}{Appendices}
\Crefname{appendix}{Appendix}{Appendices}

\usepackage[linesnumbered,boxed,ruled,vlined]{algorithm2e}

\newcommand{\Erdos}{Erd\H{o}s\xspace}
\newcommand{\Renyi}{R\'enyi\xspace}
\def\*#1{\boldsymbol{#1}}
\def\+#1{\mathcal{#1}}
\def\-#1{\mathrm{#1}}
\def\^#1{\mathbb{#1}}
\def\!#1{\mathfrak{#1}}

\newcommand{\tuple}[1]{\left(#1\right)}

\newcommand{\tp}{\tuple}

\newcommand{\abs}[1]{\left\vert#1\right\vert}

\DeclareMathOperator*{\argmax}{arg\,max}

\def\oE{\mathbb{E}}
\newcommand{\E}[2][]{ \ifthenelse{\isempty{#1}}
  {\oE\left[#2\right]}
  {\oE_{#1}\left[#2\right]} }

\DeclareMathOperator*{\oVar}{\mathbf{Var}}
\newcommand{\Var}[2][]{ \ifthenelse{\isempty{#1}}
  {\oVar\left[#2\right]}
  {\oVar_{#1}\left[#2\right]} }

\def\oEnt{\mathbf{Ent}}
\newcommand{\Ent}[2][]{ \ifthenelse{\isempty{#1}}
  {\oEnt\left[#2\right]}
  {\oEnt_{#1}\left[#2\right]} }

\newcommand{\mmse}[0]{\mathrm{MMSE}}
\newcommand{\snr}[0]{\mathrm{snr}}

\newcommand{\Dkl}{D_{\mathrm{KL}}\left( \+P \| \+Q \right)}
\newcommand{\Dalpha}{D_\alpha\left( \+P \| \+Q \right)}

\title{Recovery thresholds for hidden weighted sparse graphs
\thanks{Accepted for presentation at the Conference on Learning Theory (COLT) 2026.}
}

\usepackage{authblk}

\author[1]{Zhe Hou}
\author[1]{Jingcheng Liu}

\affil[1]{State Key Laboratory for Novel Software Technology, Nanjing University}

\affil[ ]{\texttt{houzhe@smail.nju.edu.cn}, \texttt{liu@nju.edu.cn}}

\date{}

\begin{document}

\maketitle

\begin{abstract}

Recovering structural information from noisy high-dimensional data is a fundamental task in statistical inference. We investigate the information-theoretic recovery thresholds for a graph hidden in a randomly weighted complete graph.
Specifically, an unknown graph $H^* \in \mathcal H_n$ is chosen uniformly at random, and hidden in a complete graph of $n$ vertices through edge weights: for an edge $e \in H^*$, its weight is distributed independently according to $\mathcal{P}_n$; otherwise the edge weight is distributed independently according to $\mathcal{Q}_n$. The goal is to recover the hidden edge set from these noisy observations. Our primary focus on \emph{almost exact recovery}, where all but a vanishing fraction of the edges of $H^*$ must be identified.
By choosing $\mathcal{P} = \mathrm{Bern}(1), \mathcal{Q} = \mathrm{Bern}(q)$, this model captures the well-studied planted \Erdos-\Renyi recovery setting, and other weighted formulations such as Gaussian and Exponential distributions.

Assuming a local Lipschitzness of the \Renyi divergence between distributions $\mathcal{P}_n$ and $\mathcal{Q}_n$, and a mild density condition for the graphs $\mathcal H_n$, we give a unified characterization of the information-theoretic limit for recovering almost all of $H$ (also known as almost exact recovery).

Our characterization uses the KL divergence $D_{\mathrm{KL}}(\mathcal{P}_n\|\mathcal{Q}_n)$ as the signal-to-noise metric. We show that there is a critical threshold $D_c$ governed by the logarithm of the first moment threshold of $\mathcal H_n$ in the \Erdos-\Renyi random graph model $G(n,p)$.
For any constant $\eta$, if $D_{\mathrm{KL}} \ge (1+\eta)D_c$, the maximum likelihood estimator (MLE) achieves almost exact recovery; conversely if $D_{\mathrm{KL}} \le (1-\eta)D_c$, almost exact recovery is information-theoretically impossible. Our MLE analysis is based on an extension of~\cite{ding2020consistent} ;
While our lower bound involves deriving a distance-based Fano-type inequality based on Sibson's $\alpha$-mutual information, and it also extends to the task of partial recovery, in which only a constant $\lambda$-fraction of the hidden graph is required to be recovered. 
As a result, if $D_{\mathrm{KL}} \le (\lambda-\eta)D_c$, we also proved that the recovery of any $\lambda$-fraction must fail with high probability.

Another interesting phenomenon exhibited by some of these models is the ``All-or-Nothing'' (AoN), where one can either recover almost all of the hidden graph or essentially nothing at all. For several natural distributional families, including  certain Gaussian, Bernoulli and Exponential regimes, we establish an All-or-Nothing (AoN) threshold phenomenon at the exponential scale.
For Gaussians, we obtained an AoN by lifting the sharp almost exact recovery threshold through the I-MMSE relation;
For the others, we combine a second-moment argument inspired by planted subgraph recovery~\cite{mossel2023sharp} with the second-order \Renyi divergence. 
We also provide distributional examples where AoN is not universal at this level of generality, and our partial recovery lower bound is already the best possible.

\end{abstract}

\clearpage

\section{Introduction}
\label{sec:intro}

Given a positive integer $n$, two distributions $\mathcal{P}_n$ and $\mathcal{Q}_n$, and a graph collection $\mathcal{H}_n$ in which every graph has exactly $m$ edges, we study the recovery of a hidden graph $H^*$ from a weighted complete graph $A$ on $n$ vertices and $N = \binom{n}{2}$ edges. Hereafter, we abbreviate $\mathcal{P}_n$, $\mathcal{Q}_n$, and $\mathcal{H}_n$ as $\+P$, $\+Q$, and $\+H$, respectively. First, $H^*$ is sampled uniformly from $\mathcal{H}_n$. Then each edge in $H^*$ receives an independent weight from $\mathcal{P}_n$, whereas each edge outside $H^*$ receives an independent weight from $\mathcal{Q}_n$.
The task is to recover $H^*$ from $A$ given $\+P,\+Q,\+H$. Here the hidden edge set $H^*$ is the signal, while the remaining edge weights constitute the noise.
Besides its mathematical appeal, locating the information-theoretic limit of such recovery is also motivated by applications where data collections are costly, such as in brain imaging and whole-genome sequencing.
Here, we require that $m \to \infty$ as $n \to \infty$, so that the recovery guarantees can be well-defined.
We mainly focus on the following recovery guarantees: 
\begin{itemize}[topsep=1pt, itemsep=1pt, parsep=0pt]
    \item Almost exact recovery: An estimator achieves almost exact recovery if it finds $\hat{H} \in \mathcal{H}$
such that, the number of overlapping edges between $H^*$ and $\hat{H}$ satisfies $\Pr[|\hat{H} \cap H^*| <$
$(1 - \delta_n)m] \to 0$, for some sequence $\delta_n$ that tends to $0$ as $n \rightarrow \infty$.
    \item Partial recovery: For a constant $\lambda \in (0, 1)$, the goal is to find a $\hat{H} \in \mathcal H$ such that $|\hat{H} \cap H^*| \ge \lambda m$ with probability $1-o_n(1)$.
\end{itemize}
If at any given signal-to-noise ratio, as $n \to \infty$, one could either recover almost all of $H^*$, or not even any constant fraction of $H^*$, this is known as an All-or-Nothing (AoN) phenomenon.

Depending on the choice of $\+H_n$, this model captures many statistical inference questions. When $\+H_n$ consists of all isomorphic copies of a given graph $H_n$, we refer to it as an \emph{isomorphic graph collection}, denoted by $\+H_H$. Examples include:

\begin{itemize}[topsep=1pt, itemsep=1pt, parsep=0pt]
    \item \textbf{Planted cliques.} The well-known planted clique problem~\cite{jerrum1992large} corresponds to $\+H$ being the set of $k$-cliques, and $\+P_n=\mathrm{Bern}(1)$ and $\+Q_n=\mathrm{Bern}(p)$ are Bernoulli distributions.
  Some recent examples are in complexity theory~\cite{hirahara2024planted}, cryptography~\cite{bogdanov2024low}, and algorithms such as MCMC~\cite{chen2023almost}, sum-of-squares and low-degree methods~\cite{barak2019nearly,hopkins2018statistical}. 
  \item \textbf{Gaussian hidden cliques and the sparse Wigner model.} The Gaussian variant of planted cliques has also received a lot of attention in the machine learning community both for studying spectral algorithms~\cite{montanari2015limitation}, and as the sparse spiked Wigner model for sparse PCA~\cite{deshpande2014sparse,d2020sparse,ding2024subexponential}. Indeed, in the sparse spiked Wigner model, one needs to recover a sparse unit vector $x$ from observing $\sqrt{\mathrm{snr}}xx^\top + W$, where $W$ is an independent Gaussian Wigner matrix. Notably, \cite{niles2020all} showed an \emph{All-or-Nothing phenomenon} for this model (in fact, for the general Gaussian additive model) under discrete prior: at any fixed signal-to-noise ratio, one could either recover the signal almost perfectly, or nothing at all. Under the discrete prior studied in~\cite{niles2020all}, this AoN result rules out a distinct constant-fraction partial-recovery regime.
  \item \textbf{Hamiltonian cycle} recovery was studied in~\cite{BVJ20}, motivated by genome assembly, where the hidden cycle represents the true order of genome segments and edge weights model experimental Hi-C reads~\cite{lieberman2009comprehensive}. They established sharp thresholds for exact recovery under suitable divergence conditions.
  \item \textbf{Perfect matchings} recovery was proposed in~\cite{chertkov2010inference} to model moving objects/particles tracking. Since then, many variants  
have been proposed and studied, see, e.g., ~\cite{semerjian2020recovery,moharrami2021planted} for exponential distribution on weights, and~\cite{ding2023planted} for more general distributions
and allowing the observation to be a sparse bipartite version of $G(n,d/n)$ instead of a weighted complete graph.
  \item \textbf{$2k$-nearest-neighbor} (NN) graphs, which consist of $n$ nodes on a ring, where each node is also connected to its $2k$ nearest neighbors, was introduced by~\cite{ding2020consistent} as an improved approximation to genome assembly that takes advantage of longer-range but still nearby linkage information. This naturally generalizes Hamiltonian cycles, and it can also be seen as a variant of the celebrated Watts-Strogatz small-world graph model~\cite{watts1998collective,newman1999scaling}, where a $2k$-NN graph is hidden in an unweighted random graph by a rewiring procedure.

\end{itemize}

We also consider \emph{non-isomorphic graph collections} $\+H$ with the same number of vertices and edges, that are not  generated by isomorphic copies of a single graph,
and we name a few such examples:
\begin{itemize}[topsep=1pt, itemsep=1pt, parsep=0pt]
  \item \textbf{Uniform $m$-subgraphs} are perhaps the most basic setting. We consider all subgraphs with $m$ edges, ignoring structural constraints. Then, the set of edge weights can be viewed as an $\binom{n}{2}$-dimensional real-valued vector. The recovery of a hidden $m$-subgraph is therefore a discrete analog of sparse mean estimation problems~\cite{donoho1994minimax,raskutti2011minimax}. Indeed, in the $m$-sparse Gaussian location family, the mean $\theta\in \^R^{\binom{n}{2}}$ has at most $m$ non-zeros. In these mean estimation problems, one is often more interested in the sample complexity in relation to an estimation error (often $\ell_2$). Our setting is more combinatorial in nature,  as one tries to recover the mean $\theta$ given only one sample, and the error is measured in Hamming distance.
  \item \textbf{Trees} were analyzed very recently in \cite{moharrami2025planted} via local weak convergence, using fixed-point equations to precisely characterize the fraction of correctly recovered edges by the minimum spanning tree, as well as its mean weight. Their approach is also applicable to the planted Hamiltonian cycle problem and the detection problem.
  \item \textbf{$k$-factors} can be seen as a generalization of perfect matchings. A $k$-factor is a spanning subgraph of degree $k$. Indeed, $1$-factor corresponds to perfect matchings, and a $2$-factor consists of a set of cycles that spans the entire graph, and is closely related to Hamiltonian cycles. \cite{BVJ20} show that within the exact recovery regime of the hidden Hamiltonian cycle model, a linear programming relaxation designed for $2$-factor can be used to recover the hidden Hamiltonian cycle with high probability.
\end{itemize}

\subsection{Our results}
By viewing the hidden graph as a structural ``signal'', and the noisy edge weights as the ``noise'', a natural question is:
\begin{quote}
  {\em What is the ``signal-to-noise ratio'' threshold for almost exact recovery of a structural ``signal''? And if one only needs to recover a $\lambda$ fraction of the hidden graph, how does the threshold change in terms of $\lambda$?}
\end{quote} 

We give a unified answer that applies to a broad class of graphical structures and a broad class of distributions, by connecting the KL divergence between $\+P$ and $\+Q$, to the first moment threshold in \Erdos-\Renyi random graph models.
Our characterization applies to collections of graphs $\+H$ that are \emph{uniformly sparse}, and to distributions whose \Renyi divergence satisfies a local Lipschitzness condition. In the following, we explain these conditions on $\+H$ and assumptions on distributions $\+P$ and $\+Q$ more formally.

{\noindent\textbf{Assumptions on graph collections.}}
Before introducing the notion of \emph{uniformly sparse} graphs, it is helpful to recall the \emph{first moment flat} condition introduced by~\cite{mossel2023sharp}, under which they characterized the All-or-Nothing (AoN) threshold in \Erdos-\Renyi random graph model with a planted graph from isomorphic $\+H$. They further showed that first moment flatness is also necessary for many graph collections in the planted \Erdos-\Renyi random graph model to exhibit AoN.
Intuitively, this condition asks that any $H\in \+H$ does not contain any particularly dense subgraph.
More formally, for a graph collection $\mathcal{H}$, let $M=\abs{\+H}$ denote its cardinality and $m$ denote the number of edges in each graph, the \emph{first moment threshold} is defined as $p_{1M}(\+H) := M^{-\frac{1}{m}}$. Put differently, it is the threshold where the expected number of copies of all $H \in \+H$ in \Erdos-\Renyi graph $G(n,p)$ equals $1$.
For an isomorphic graph collection $\mathcal{H}_H$ consisting of all graphs isomorphic to graph $H$, we further define the \emph{expectation threshold} as $p_E(\mathcal{H}_H) := \max_{J \subseteq H} p_{1M}(\mathcal{H}_J)$, where $J$ is an edge induced subgraph of $H$ (hereafter abbreviated as $p_{1M}$ and $p_E$ respectively). These are basic thresholds in random graph theory that serve as convenient lower bounds for the critical threshold.
\begin{definition}[First Moment Flat]
\label{def:first-moment-flat}
An isomorphic graph collection $\mathcal{H}_H$ is \emph{first moment flat} if 
\begin{equation}
\log (p_E(\mathcal{H}_H)^{-1}) = (1 + o(1)) \log (p_{1M}(\mathcal{H}_H)^{-1}).
\end{equation}
\end{definition}

The more general condition for graph collections used in our analysis is the following \emph{uniformly sparse} condition. This is closely related to the growth condition for exponential scale AoN in Equation (3.6) of \cite{mossel2023sharp}, but to handle non-isomorphic graph collections and general distributions, we adopt a slightly stronger form.

\begin{definition}[Uniformly sparse]
  \label{def:uniformly-sparse}
  A graph collection $\mathcal{H}$ is said to be \emph{uniformly sparse} if $\log(p_{1M}^{-1}) = \omega(1)$ and
\begin{equation}
\max_{H_1\in\+H}\ \max_{\ell\in[0,m]\cap\mathbb Z}
\left\{
\log \Pr_{H_2 \sim \+H}[|H_1 \cap H_2| = \ell]
+\ell \log(p_{1M}^{-1})
\right\}
= o(\log M).
\end{equation}

\end{definition}

Intuitively, the condition $\log(p_{1M}^{-1})=\omega(1)$ requires the graph to be sufficiently sparse, while the overlap probability condition rules out the presence of locally dense subgraphs.
We will show in~\Cref{sec:uniformly-sparse} that first moment flat graphs satisfying $\log(p_{1M}^{-1}) = \omega(1)$ are always uniformly sparse. Furthermore, another sufficient condition for uniform sparsity is that the graph collection is \emph{sufficiently large}. Examples include uniform $m$-subgraphs, trees, and $k$-factors.

{\noindent\textbf{Assumptions on distributions.}}
Let $D_{\alpha}(\+P \| \+Q)
    := \frac{1}{\alpha-1} \log \mathbb{E}_{\+P} \left[ \left( \frac{\mathrm{d} \+P}{\mathrm{d} \+Q} \right)^{\alpha - 1} \right]$ be the \Renyi divergence of order $\alpha$ (in this paper we assume that $\+P \ll \+Q$).
    This family of divergences interpolates several classical discrepancy measures: in particular, $\lim_{\alpha\to 1} D_{\alpha}(\+P\|\+Q)=D_{\mathrm{KL}}(\+P\|\+Q)$, and $D_{2}(\+P\|\+Q)=\log\bigl(1+\chi^{2}(\+P\|\+Q)\bigr)$, where $\chi^{2}(\+P\|\+Q)$ denotes the chi-squared divergence, a natural quantity in second-moment arguments underlying contiguity and testing lower bounds.
We focus on distributions with a \emph{stable} \Renyi divergence, which are satisfied by many natural distributions. 
These assumptions facilitate a unified analysis framework: they can be used both in deriving upper and lower bounds.

\begin{assumption}[\Renyi Stability Upper Bound]
\label{asm:renyi-stable-upper}
  For any fixed $\theta > 0$,
      there exists $\alpha_n \in (0,1)$ with $\alpha_n \to 1$ and
      $1-\alpha_n=\omega(\frac{1}{D_{\mathrm{KL}}(\+P \| \+Q)})$, such that, for all sufficiently large $n$,
  \begin{equation}
    D_{\alpha_n}(\+P \| \+Q) \ge (1-\theta) D_{\mathrm{KL}}(\+P \| \+Q),
  \end{equation}

\end{assumption}

\begin{assumption}[\Renyi Stability Lower Bound]
\label{asm:renyi-stable-lower}
    For any fixed $\theta > 0$, 
    there exists $\beta_n > 1$ with $\beta_n \to 1$ and $\beta_n-1 = \Omega(\frac{1}{D_{\mathrm{KL}}(\+P \| \+Q)})$, such that, for all sufficiently large $n$,
  \begin{equation}
    D_{\beta_n}(\+P \| \+Q) \le (1+\theta) D_{\mathrm{KL}}(\+P \| \+Q).
  \end{equation}
\end{assumption}
We will refer to distributions $\+P$ and $\+Q$ that satisfy these assumptions as \emph{\Renyi-stable}.
In \Cref{sec:renyi-stability}, we show that the above assumptions hold for many distributions (e.g., Bernoulli, Gaussian, and Exponential distributions), and in various regimes. We also provide some intuitions there.

Our main result is that the critical threshold for almost exact recovery is around $D_{\mathrm{KL}}(\mathcal{P} \| \mathcal{Q}) \approx \log(p_{1M}^{-1})$. Recovering a $\lambda$-fraction of edges requires $D_{\mathrm{KL}}(\mathcal{P} \| \mathcal{Q}) \gtrapprox \lambda \log(p_{1M}^{-1})$. Furthermore, when $D_{2}(\mathcal{P} \| \mathcal{Q})$ is below $\log(p_{1M}^{-1})$, constant fraction recovery becomes impossible.

\begin{theorem}\label{thm:main}
  Let $\mathcal{H}$ be uniformly sparse, and $\+P, \+Q$ be \Renyi-stable. It holds that
  \begin{enumerate}
    \item If $D_{\mathrm{KL}}(\mathcal{P} \| \mathcal{Q}) \ge (1 + \eta) \log(p_{1M}^{-1})$ for some constant $\eta > 0$, then the maximum likelihood estimator (MLE) achieves almost exact recovery.
    \item Conversely, almost exact recovery requires that $D_{\mathrm{KL}}(\mathcal{P} \| \mathcal{Q}) \ge (1 - o(1)) \log(p_{1M}^{-1})$.
    \item Furthermore, if $D_{\mathrm{KL}}(\mathcal{P} \| \mathcal{Q}) \le (\lambda - \eta) \log(p_{1M}^{-1})$ for some constant $0 < \eta < \lambda < 1$, then any algorithm (efficient or not) designed to recover a $\lambda$-fraction of the hidden edges fails with high probability.
    \item If $D_{2}(\mathcal{P} \| \mathcal{Q}) \le (1 - \eta) \log(p_{1M}^{-1})$ for some constant $\eta > 0$, then any algorithm (whether efficient or not) designed to recover any constant fraction of the hidden edges fails with high probability.
  \end{enumerate}
\end{theorem}
The four parts of \Cref{thm:main} are proved in \Cref{sec:almost-upper-bound,sec:almost-lower-bound,sec:partial-lower-bound,sec:12renyi-aon} respectively.

\begin{figure}[H]
\centering
\begin{tikzpicture}[scale=1.2]
\definecolor{lightred}{RGB}{255,200,200}
\definecolor{verylightred}{RGB}{255,230,230}
\definecolor{lightgreen}{RGB}{230,245,230}
\definecolor{lightblue}{RGB}{230,230,255}

\pgfmathsetmacro{\X}{8}
\pgfmathsetmacro{\Y}{3}
\pgfmathsetmacro{\A}{5.7}
\pgfmathsetmacro{\N}{2.4}

\coordinate (O) at (0,0);
\coordinate (X) at (\X,0);
\coordinate (Y) at (0,\Y);
\coordinate (A) at (\A,0);
\coordinate (N) at (\N,0);

\pgfmathsetmacro{\Factor}{\N/\A}
\coordinate (D) at ($(O)!\Factor!(\A,\Y)$);

\fill[lightred] (O) rectangle ($(N) + (Y)$);
\fill[verylightred] (D) -- ($(A) + (Y)$) -- ($(N) + (Y)$) -- cycle;

\fill[lightgreen] (A) rectangle ($(X) + (Y)$);

\fill[lightblue] (N) -- (D) -- ($(A) + (Y)$) -- (A) -- cycle;

\draw[dashed] (O) -- (D);
\draw (N) -- (D) -- ($(A) + (Y)$) -- ($(X) + (Y)$);

\pgfmathsetmacro{\U}{1.3}
\node at ($(0.5*\N, \U)$) {Nothing};
\node at ($(0.5*\A+0.5*\X,\U)$) {All};
\node at ($(\N-0.345,\Y-0.3)$) {Impossible region};
\node at ($(0.5*\A+0.5*\X,\Y-0.3)$) {Possible region};
\node[align=center] at ($(0.5*\A+0.5*\N,\U)$) {Depends on $\+P,\+Q$};

\draw[->] (O) -- ($(X) + (0.5, 0)$) node[right] {$D_{\mathrm{KL}}$};
\draw[->] (O) -- ($(Y) + (0, 0.5)$) node[above, align=center] {Target edge recovery fraction};

\draw (\N,0.1) -- (\N,-0.1);
\draw (\A,0.1) -- (\A,-0.1);
\draw (0.1,\Y) -- (-0.1,\Y);

\node[below, align=center] at (N) {$\kappa_2$\\($D_2(\mathcal{P}\|\mathcal{Q}) = \log(p_{1M}^{-1})$)};
\node[below] at (A) {$\log(p_{1M}^{-1})$};

\node[left] at (Y) {1};
\node[left] at (O) {0};
\end{tikzpicture}
\caption{Asymptotic visualization of \Cref{thm:main}. Here $\kappa_2$ denotes the $D_{\mathrm{KL}}$-coordinate at which $D_2(\mathcal P\|\mathcal Q)=\log(p_{1M}^{-1})$ along a fixed parametrization of the distribution pair.}
\label{fig:tikz-main}
\end{figure}
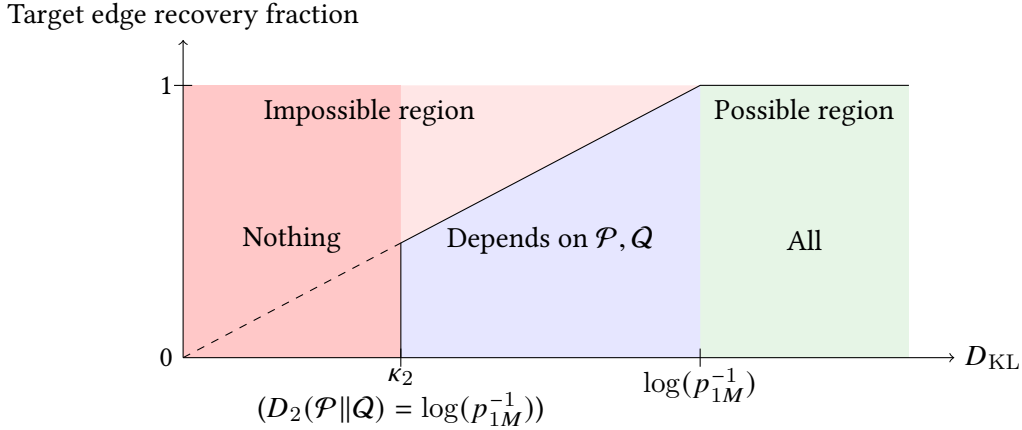

\Cref{fig:tikz-main} illustrates the recovery phase transition characterized in \Cref{thm:main}. The horizontal axis represents the KL divergence $D_{\mathrm{KL}}(\mathcal{P}\|\mathcal{Q})$, while the vertical axis represents the fraction of edges attempted to be recovered.
 The green region labeled `All' corresponds to part (1), where $D_{\mathrm{KL}} \ge (1+\eta)\log(p_{1M}^{-1})$ and the MLE achieves almost exact recovery.
 The big red upper triangular region corresponds to part (3), where $D_{\mathrm{KL}} \le (\lambda-\eta)\log(p_{1M}^{-1})$ and any algorithm fails to recover $\lambda$ fraction of edges with any constant probability.
 For a fixed parametrization of the distribution pair, let $\kappa_2$ denote the value of $D_{\mathrm{KL}}$ at which $D_2(\mathcal P\|\mathcal Q)=\log(p_{1M}^{-1})$.
 By the monotonicity of \Renyi divergence ($D_{\mathrm{KL}} = D_1 \le D_2$), $\kappa_2$ is no larger than $\log(p_{1M}^{-1})$.
 To the left of this point, the red rectangular area labeled `Nothing' represents part (4), where $D_2(\mathcal{P}\|\mathcal{Q}) \le (1-\eta)\log(p_{1M}^{-1})$, and any algorithm fails to recover any constant fraction of edges with any constant probability.
 While the regions defined by parts (3) and (4) partially overlap, their union gives the region in which the success probability of any recovery algorithm vanishes asymptotically.

For the all-or-nothing phenomenon (AoN), we adopt the definition of \cite{mossel2023sharp}. They identified two types of AoN, to unify treatment, we use the KL divergence as a common signal-to-noise scale: In the linear scale AoN regime, the phase transition is said to occur at $D_{\mathrm{KL}} = D_{\mathrm{AoN}} \pm \varepsilon$; whereas in the exponential scale regime, the transition occurs at $D_{\mathrm{KL}} = (1 \pm \varepsilon) D_{\mathrm{AoN}}$ for any constant $\varepsilon > 0$. Unless otherwise specified, we will focus on the exponential-scale AoN regime in what follows.

Unlike almost exact recovery, 
 the existence of AoN crucially depends on the choice of distribution. 
 In particular, \cite{mossel2023sharp} proved that for $\mathcal{P} = \mathrm{Bern}(1)$ and $\mathcal{Q} = \mathrm{Bern}(p)$, there is an exponential scale AoN for $\mathcal{H}$ that are first moment flat. For $\mathcal{P} = \mathrm{Exp}(1/\mu)$, $\mathcal{Q} = \mathrm{Exp}(1/n)$, \cite{moharrami2025planted} suggests that for $\mathcal{H}$ being the set of trees or Hamiltonian paths, an AoN occurs in the exponential scale. If we restate these thresholds in terms of the KL divergence, they both occur at $D_{\mathrm{KL}}(\mathcal{P} \| \mathcal{Q}) = (1\pm o(1)) \log(p_{1M}^{-1})$.

Given item 1 and 4 in \Cref{thm:main}, if $D_{2}(\mathcal{P} \| \mathcal{Q}) = (1+ o(1)) D_{\mathrm{KL}}(\mathcal{P} \| \mathcal{Q})$, then one immediately has an exponential scale AoN at $D_{\mathrm{KL}}(\+P\|\+Q) \approx \log(p_{1M}^{-1})$.
 
 \begin{corollary}\label{col:12renyi}
    Let $\+H$ be uniformly sparse, and let $\+P,\+Q$ be \Renyi-stable. If
    $D_{2}(\mathcal{P} \| \mathcal{Q}) = (1+ o(1)) D_{\mathrm{KL}}(\mathcal{P} \| \mathcal{Q})$, then
    \begin{enumerate}
    \item if $D_{\mathrm{KL}}(\mathcal{P} \| \mathcal{Q}) \ge (1 + \eta) \log(p_{1M}^{-1})$ for some constant $\eta > 0$, then MLE achieves almost exact recovery.
    \item If $D_{\mathrm{KL}}(\mathcal{P} \| \mathcal{Q}) \le (1 - \eta) \log(p_{1M}^{-1})$ for some $\eta>0$, then any algorithm (whether efficient or not) designed to recover any constant fraction of the hidden edges fails with high probability.
  \end{enumerate}
\end{corollary}

Not all distributions satisfy $D_{2}= (1+ o(1)) D_{\mathrm{KL}}$. In \Cref{sec:12renyi-check}, we show that distributions satisfying $D_{2}= (1+ o(1)) D_{\mathrm{KL}}$ include Bernoulli and Exponential distributions under \emph{natural parameterizations}. 
Two notable exceptions are $\+P=\-{Bern}(\Theta(1)),\ \+Q=\-{Bern}(o(1))$, and Gaussian distributions.

For the first exception of Bernoulli, there is no AoN in general, and our partial recovery lower bound (item 3 in 
\Cref{thm:main}) is tight. In this case, the entire blue region of~\Cref{fig:tikz-main} can have partial recovery.
\begin{theorem}\label{thm:non-aon}
  Fix any $\lambda\in (0,1)$, and assume $m=o(N)$ and $m\to\infty$. Let $\+H$ be the collection of all $m$-edge subgraphs of $K_n$, and let $\+P=\-{Bern}(\lambda), \+Q=\-{Bern}\left(\frac{(1-\lambda) m}{N-m}\right)$. Then $D_{\mathrm{KL}}(\mathcal{P} \| \mathcal{Q}) = (\lambda+o(1)) \log(p_{1M}^{-1})$ and
	    there exists an estimator that recovers a $\lambda + o(1)$ fraction of the hidden edges with high probability.
\end{theorem}
We prove this theorem in \Cref{sec:partial-tightness}.

 For the other exception of Gaussian distributions, AoN has been studied in the context of sparse PCA~\cite{niles2020all}.
 We show that AoN is universal for recovering hidden structures with Gaussian weights, by lifting the sharp threshold for almost exact recovery to an AoN through the I-MMSE relation for Gaussian channels. In particular, this means that the entire blue region of~\Cref{fig:tikz-main} corresponds to ``Nothing'' for Gaussian distributions.
\begin{theorem}\label{thm:aon}
  Let $\mathcal{H}$ be uniformly sparse, $\+P = N(\mu_n, 1)$, $\+Q = N(0, 1)$, and $\mu_n > 0$. For every fixed $\eta>0$, if
  $D_{\mathrm{KL}}(\+P\|\+Q)\ge (1+\eta)\log(p_{1M}^{-1})$, then the MLE achieves almost exact recovery. If
  $D_{\mathrm{KL}}(\+P\|\+Q)\le (1-\eta)\log(p_{1M}^{-1})$, then for every fixed $\lambda,\delta \in (0,1)$, no estimator can recover a $\lambda$-fraction of the hidden edges with probability at least $\delta$.
\end{theorem}

Last but not least, we show in~\Cref{sec:topm} that for \emph{sufficiently large} $\+H$ and distributions satisfying \Cref{asm:renyi-stable-upper}, a simple top-$m$ selection based on the likelihood ratio recovers almost all hidden edges with high probability, if the estimator is allowed to output an arbitrary $m$-edge subset rather than a member of $\+H$. Notable examples of \emph{sufficiently large} $\+H$ include perfect matchings and Hamiltonian paths/cycles  (see~\Cref{tbl:p2}).

\subsection{Technical overview and paper organization}
The notion of uniform sparsity will play a recurring role throughout our analysis, as it allows a unified analysis of a broad class of structures, whether isomorphic or not. Two convenient sufficient conditions for uniform sparsity are \emph{first moment flat} and  \emph{sufficiently large}. While our main focus is information-theoretic thresholds for estimators taking values in $\+H$, we note that for sufficiently large $\+H$ and distributions satisfying \Cref{asm:renyi-stable-upper}, a simple top-$m$ selection already recovers almost all hidden edges if arbitrary $m$-edge subsets are allowed.

Our upper bound for almost exact recovery in \Cref{thm:main} comes from the maximum likelihood estimator (MLE). We extend the approach of \cite{ding2020consistent}, which was designed for $2k$-NN graphs and for distributions whose log-MGF has a quadratic lowerbound near the boundary. We are able to generalize to a broad class of graphs -- \emph{uniformly sparse} graph collections, and to a wider family of \Renyi-stable distributions. In essence, we use Chernoff bound together with union bound to control the deviations of likelihood. In particular, the \emph{optimal} Chernoff exponent comes from balancing the deviations of edges inside and outside the hidden $H$. The first moment threshold dictates what Chernoff exponent can beat the union bound,  while the KL divergence dictates what Chernoff exponent can be achieved.
This is established in~\Cref{sec:almost-upper-bound}.

In \Cref{sec:it-lower-bounds}, we employ two information-theoretic methods: Our almost exact recovery lower bound (part (2) in \Cref{thm:main}) is based on a standard application of distance-based Fano's inequality, and is established in~\Cref{sec:almost-lower-bound}. Notably, this simple lower bound does not require the assumption of \Renyi-stability of $\+P,\+Q$. The key to a tight characterization is a good estimate for the mutual information, and controlling the growth rate of overlaps for uniformly sparse $\+H$.
For partial recovery, however, standard Fano's inequality based on mutual information does not yield a strong enough converse bound. Our partial recovery lower bound (part (3) in \Cref{thm:main}) is based on working with the Sibson's $\alpha$-mutual information instead of the mutual information. We establish a distance-based variant of Fano's inequality from Sibson's $\alpha$-mutual information \cite{esposito2024sibson}, from which we derive our lower bound by choosing an appropriate $\alpha$ in \Cref{sec:partial-lower-bound}. The existence of such $\alpha\neq 1$ is guaranteed by \Renyi stability.

\Cref{sec:aon} focuses on the lower bound for the `Nothing' phase and  sufficient conditions for AoN. \Cref{sec:12renyi-aon} generalizes the second moment method from \cite{mossel2023sharp} to weighted distributions, and substituting the corresponding ``signal-to-noise ratio'' part with $D_2$. This proves part (4) of \Cref{thm:main}, demonstrating AoN for distributions with $D_2=(1+o(1))\Dkl(\+P\|\+Q)$. For our Gaussian AoN result in \Cref{thm:aon},  we show in \Cref{sec:gaussian-aon} that a sharp transition for almost exact recovery can be lifted to an AoN phenomenon with the help of an I-MMSE relation. I-MMSE relation (restated at \Cref{lem:I-mmse}) is most famously known for Gaussian distributions, and it would be interesting to see if there could be relaxation of the I-MMSE relation where such lifting is possible. 

Last but not least, we also complement with an example where our partial recovery lower bound is tight in~\Cref{sec:partial-tightness}.

\subsection{Related works}
There has been significant recent progress on the problems of almost exact recovery and partial recovery. The most important part of our work is a unified approach across a broad class of graph structures, or a broad class of distributions, or both. In specific instances however, there are some prior works that give more refined phase transition than ours. We explain these instances where we fall short below.

\cite{ding2020consistent} studied almost exact recovery in $2k$-NN graphs under certain assumptions on the distribution. 
They established a threshold on the KL divergence of the distribution: when $D_{\mathrm{KL}}(\+P \| \+Q) = (1 + \eta) \log(n) / k$, almost exact recovery is achievable, and they also proved that this threshold is information-theoretically necessary. Our result for almost exact recovery is a generalization both in terms of the hidden structure, and in terms of the distributions $\+P$ and $\+Q$.

\cite{mossel2023sharp} investigated recovery in the classical \Erdos-\Renyi random model $G(n,p)$. They studied both linear scale and exponential scale AoN and demonstrated that many planted subgraph recovery problems exhibit AoN.
The requirement they impose on graphs to exhibit exponential scale AoN is precisely first moment flatness, together with either $\log(p_{1M}^{-1})=\omega(1)$ or $v(H)=n^{o(1)}$, and they characterized linear scale AoN for graph families that satisfy more stringent structural conditions.
While \cite{mossel2023sharp} showed AoN for $\+P$ and $\+Q$ being Bernoulli distributions, we show that under a different parameterization of Bernoullis, there is no AoN. In fact, our partial recovery lower bound can be sharp.

\cite{moharrami2021planted} investigated the planted matching problem under exponential distributions on edge weight. In the All phase, they establish a linear scale phase transition, while in the Nothing phase, they employ the method of local weak convergence and analyze an associated system of ordinary differential equations (ODEs) to derive exact expressions for the expected overlap. After that, \cite{ding2023planted} provides an explicit lower bound and confirms the infinite-order phase transition conjectured in \cite{semerjian2020recovery}. Subsequently, \cite{moharrami2025planted} studied the problem of recovering a planted Hamiltonian cycle or a spanning tree using the minimum spanning tree (MST). 
They derived, via a fixed-point equation, the asymptotic relation between a constant parameter $\mu$ and the fraction of edges recovered. Rephrased in terms of the KL divergence, it corresponds to
$D_{\mathrm{KL}}(\mathcal{P}\|\mathcal{Q}) = (1+o(1)) \log(p_{1M}^{-1}/\mu)$,
which implies an AoN transition on the exponential scale, while there is no AoN in the linear scale: partial recovery remains possible and the fraction of recovery varies smoothly with $\mu$.

\begin{table}[h!]
\caption{Summary of works with graph structures, distributions, and properties.}
\centering
\begin{tabular}{|c|c|c|c|}
\hline
Work & Graph Structure & Distribution & Properties \\ \hline
\cite{ding2020consistent} & $2k$-NN graphs & \makecell{Gaussian /\\ Assumption 2 of \cite{ding2020consistent}} & almost exact recovery \\ \hline
\cite{mossel2023sharp} & \makecell{first moment flat\\ and either $\log(p_{1M}^{-1})=\omega(1)$\\ or $v(H)=n^{o(1)}$} & $\mathcal{P} = \mathrm{Bern}(1),\mathcal{Q} = \mathrm{Bern}(q)$ & exponential AoN \\ \hline
\cite{mossel2023sharp} & \makecell{sufficiently dense\\delocalized\\almost balanced} & $\mathcal{P} = \mathrm{Bern}(1),\mathcal{Q} = \mathrm{Bern}(q)$ & linear AoN \\ \hline
\makecell{\cite{moharrami2021planted}\\ \cite{ding2023planted}} & perfect matchings & $\mathcal{P} = \mathrm{Exp}(1/\mu),\mathcal{Q} = \mathrm{Exp}(1/n)$ & \makecell{almost exact recovery\\ in linear scale} \\ \hline
\cite{moharrami2025planted} & \makecell{spanning trees and\\ Hamiltonian paths} & $\mathcal{P} = \mathrm{Exp}(1/\mu),\mathcal{Q} = \mathrm{Exp}(1/n)$ &  exponential AoN \\ \hline
Our & uniformly sparse & Under \Cref{asm:renyi-stable-upper} & almost exact recovery \\ \hline
Our & uniformly sparse & \makecell[l]{1. $\+P = N(\mu, 1)$, $\+Q = N(0, 1)$\\
2. $\+P=\mathrm{Bern}(1-o(1))$,\\ \quad $\+Q=\mathrm{Bern}(o(1))$\\ 3. $\+P=\mathrm{Exp}(\lambda_1)$,\\\quad  $\+Q=\mathrm{Exp}(\lambda_2)$ with $\lambda_1 \gg \lambda_2$ } & exponential AoN \\ \hline
Our & uniformly sparse & $\+P = \mathrm{Bern}(\lambda)$, $\+Q = \mathrm{Bern}(\frac{(1-\lambda) m}{N-m})$ & no AoN \\ \hline
\end{tabular}
\label{tab:summary}
\end{table}
\vspace{-10pt}

\subsection{Discussions and open problems}

Although our work unifies a broad class of hidden weighted graph recovery problems across various graph structures and edge weight distributions, several limitations and open problems remain. In particular, beyond uniformly sparse graph families, a complete understanding is still lacking. 

In this paper, we mainly focus on the study of almost exact recovery and partial recovery. One may be tempted to ask for a unified understanding of the exact recovery thresholds. We note, however, that it would likely require a more detailed study of a local structure, which may vary across different graph structures, such as the ``crossing'' structure in \cite{BVJ20}.
Likewise, a unified understanding of AoN would also be desirable. On the distributional side, it remains unclear for which pairs of distributions $(\mathcal{P},\mathcal{Q})$ AoN should occur: at present, we only identify two independent sufficient, but not necessary, conditions that guarantee AoN. On the graph-structural side, \cite{mossel2023sharp} further shows that, under the $G(n,p)$ planted model, for two broad class of graph collections, they established necessary and sufficient conditions for AoN to hold at linear or exponential scales. Furthermore, a very recent work \cite{lee2025fundamental} investigated more general graph structures under the $G(n,p)$ planted model and analyzed their MMSE curves in regimes where AoN may fail to occur.

AoN also appears to be intimately tied to computational-statistical gaps in many statistical inference problems. We focus mainly on the statistical limits of recovery in this paper, and it would be nice to have a unified understanding of where AoN, or computational-statistical gaps could be found.

\section{Almost exact recovery upper bound by MLE}
\label{sec:almost-upper-bound}

We adapt the framework of \cite{ding2020consistent} to uniformly sparse $\+H$ and \Renyi-stable distributions.

Since $H^*$ is chosen uniformly at random from $\+H$, the MLE coincides with the maximum a posteriori (MAP) estimator. Let $\hat{H}_{\mathrm{MLE}}$ denote the graph in $\+H$ that maximizes the likelihood. By definition:
\begin{equation}
  \begin{aligned}
    \hat{H}_{\mathrm{MLE}} &= \argmax_{H \in \mathcal{H}} \prod_{e \in H} \frac{\mathrm{d}\mathcal{P}}{\mathrm{d}\mathcal{Q}}(A_e) = \argmax_{H \in \mathcal{H}} \sum_{e \in H} L_e,
  \end{aligned}
\end{equation}
where $A_e$ is the weight of edge $e$ in the observed graph $A$, and
$L_e := \log \frac{\mathrm{d}\mathcal{P}}{\mathrm{d}\mathcal{Q}} (A_e)$ represents the log-likelihood ratio of $A_e$. Since $\mathcal{P} \ll \mathcal{Q}$, this Radon--Nikodym derivative is well-defined, and $L_e$ is understood as an extended-real random variable. Let $\+X$ and $\+Y$ denote the distributions of $\log \frac{\mathrm{d}\mathcal{P}}{\mathrm{d}\mathcal{Q}}$ under the measures $\mathcal{P}$ and $\mathcal{Q}$, respectively. Then, for any edge $e$, the log-likelihood $L_e$ independently follows the distribution $\+X$ if $e \in H^*$, and follows the distribution $\+Y$ otherwise. 

Define $L(H) := \sum_{e \in H} L_e$. The MLE selects a graph $\hat{H}_{\mathrm{MLE}}$ that maximizes $L(H)$. We aim to show that, with high probability, every maximizer is close to the hidden graph $H^*$. In fact, we will prove a slightly stronger statement: there exists a sequence $\delta_n \to 0$ such that, with high probability, no graph $H \in \+H$ whose distance to $H^*$ exceeds $\delta_n m$ satisfies $L(H) > L(H^*)$ (the distance is defined as the number of edges that appear in $H^*$ but not in $H$).

Formally, the above analysis can be written as:
\begin{equation}\label{eq:MLE-1}
  \begin{aligned}
    \Pr\Bigl[|H^* \setminus \hat{H}_{\mathrm{MLE}}| > \delta_n m\Bigr] &\le \Pr\Bigl[\exists H \in \+H, |H^* \setminus H| > \delta_n m,\ L(H) > L(H^*)\Bigr].
  \end{aligned}
\end{equation}
Since every graph in $\+H$ has $m$ edges, $|H^*\setminus H|=|H\setminus H^*|$. Define $\+H(H^*,\ell) := \{H:|H\setminus H^*| = \ell, H \in \+H\}$. By the union bound, we have:
\begin{equation}\label{eq:MLE-Union}
  \begin{aligned}
    \Pr\Bigl[\exists H \in \+H, |H^* \setminus H| > \delta_n m,\ L(H) > L(H^*)\Bigr] &\le \sum_{\ell=\lceil \delta_n m\rceil}^m \Pr\Bigl[\exists H \in \+H(H^*,\ell): L(H) > L(H^*)\Bigr].
  \end{aligned}
\end{equation}

\begin{lemma}\label{lem:almost-upper}
  For any uniformly sparse $\+H$, when $D_{\mathrm{KL}}(\+P \| \+Q) \ge (1+\eta) \log(p_{1M}^{-1})$ for some constant $\eta > 0$, there exist sequences $\epsilon_n, \delta_n \to 0$ such that:
  \begin{equation}\label{eq:MLE-chernoff}
  \begin{aligned}
    \Pr\Bigl[&\exists H \in \+H(H^*,\ell): L(H) > L(H^*)\Bigr]\\
    &\le 
    \exp\left( -(1+o(1))\frac{\ell\eta}{4\epsilon_n}\right) +
    \exp\left( -\frac{\ell \eta \log(p_{1M}^{-1})}{4} + o(\log M)\right)\quad \forall \ell \in [\delta_n m, m]\\
  \end{aligned}
  \end{equation}
\end{lemma}

We defer this proof to \Cref{sec:missing-proofs3}.

{\noindent\textbf{Proof of the almost exact recovery upper bound in \Cref{thm:main}}.}
Combining~\cref{eq:MLE-1},~\cref{eq:MLE-Union} and~\cref{eq:MLE-chernoff}, we have:
\begin{equation}
  \begin{aligned}
    \Pr\Bigl[|H^* \setminus \hat{H}_{\mathrm{MLE}}| > \delta_n m\Bigr] 
    &\le \sum_{\ell = \lceil \delta_n m\rceil}^{m} \Pr\Bigl[\exists H \in \+H(H^*,\ell): L(H) > L(H^*)\Bigr]\\
    &\le \sum_{\ell = \lceil \delta_n m\rceil}^{\infty} \exp\left( -(1+o(1))\frac{\ell\eta}{4\epsilon_n}\right) + \exp\left( -\frac{\ell \eta}{4} \cdot \log(p_{1M}^{-1}) + o(\log M)\right)\\
    &= \frac{\exp(- (1+o(1))\frac{\delta_n m \eta}{4\epsilon_n}) }{1 - \exp(-(1+o(1))\frac{\eta}{4\epsilon_n})} + \frac{\exp(-\frac{\delta_n m \eta\log(p_{1M}^{-1})}{4} + o(\log M))}{1 - \exp(-\frac{\eta\log(p_{1M}^{-1})}{4})}
    = o(1).
    \end{aligned}
\end{equation}
Thus, the MLE achieves almost exact recovery.
\qed

\section{Information-theoretic recovery lower bounds}\label{sec:it-lower-bounds}

In this section, we use information-theoretic tools to analyze the necessary conditions for almost exact recovery and partial recovery. For uniformly sparse $\mathcal{H}$, our lower bound for almost exact recovery coincides with the upper bound established in \Cref{sec:almost-upper-bound}. For partial recovery, the performance varies significantly across different distributions. While some distributions exhibit a sharp AoN phase transition, some other distributions show a phenomenon where the maximum overlap aligns with our partial recovery lower bound, decreasing linearly with the KL divergence.

\subsection{Almost exact recovery lower bound by distance-based Fano's inequality}
\label{sec:almost-lower-bound}

Our information-theoretic lower bound for almost exact recovery is based on an application of the distance-based Fano’s inequality.
To do so, inference problems are typically modeled as follows: the \emph{parameter} $X$ is drawn randomly from a set $\mathcal{X}$, the \emph{observation} $Y$ and the \emph{estimator} $\hat{X}$ form a Markov chain $X \rightarrow Y \rightarrow \hat{X}$.
The marginal distributions of $X$ and $Y$ are denoted by $P_X$ and $P_Y$, respectively, and their joint distribution is denoted by $P_{XY}$.

More specifically, in our setting, $\+X = \+H$ and the parameter $X$ is a graph sampled uniformly at random from $\+X$, the observation $Y$ is a weighted complete graph, where the edge weights in $Y$ are generated conditionally on $X$, such that edges in $X$ i.i.d. follow distribution $\mathcal{P}$, and all other edges i.i.d. follow $\mathcal{Q}$.

\begin{lemma}\label{lem:reverse-fano}
  For uniformly sparse $\+H$, if there exists an algorithm that achieves almost exact recovery, then it must hold that 
  \begin{equation}\label{eq:reverse-fano}
    I(X; Y) \ge (1 - o(1)) \log M.
  \end{equation}
\end{lemma}
We defer this proof to \Cref{sec:missing-proofs4}.

\begin{lemma}\label{lem:mutual-information-upperbound}
  The mutual information has following upper bound in hidden weighted graph setting.
  \begin{equation}\label{eq:KL-information-upper-bound}
    I(X; Y) \le m \cdot D_{\mathrm{KL}}(\mathcal{P} \| \mathcal{Q}).
  \end{equation}
\end{lemma}
We defer this proof to \Cref{sec:missing-proofs4}.

{\noindent\textbf{Proof of the almost exact recovery lower bound in \Cref{thm:main}}.}

Combining \Cref{eq:reverse-fano} with \Cref{eq:KL-information-upper-bound} finishes the proof.
\qed

\subsection{Partial recovery lower bound using Sibson's \texorpdfstring{$\alpha$}{alpha}-Mutual Information}
\label{sec:partial-lower-bound}

The standard Fano's inequality does not yield a strong enough converse bound. For example, when the mutual information between the parameter and the observation is only a constant fraction of the parameter's entropy, the inequality can at best conclude that the probability of partial recovery is no more than another constant. Therefore, we need stronger measures than mutual information to capture the fundamental limits of recovery more precisely in the low-information regime.

In \cite{esposito2024sibson}, the authors studied Sibson's $\alpha$-mutual information and proposed Fano-type inequalities based on Sibson's mutual information. We need a distance-based variant as follows:
\begin{lemma}\label{lem:sibson-fano}
  For any $\alpha>1$ and $X$ distributed over a finite set $\mathcal{X}$, let $\rho: \mathcal{X} \times \mathcal{X} \to \mathbb{R}_{\ge 0}$ be a symmetric distance function. For any estimator $\hat{X}$ that forms a Markov chain $X \rightarrow Y \rightarrow \hat{X}$, we have:
  \begin{equation}
    \Pr[\rho(X, \hat{X}) \le t] \le \left( p^\star_t e^{I_\alpha(X;Y)} \right)^{\frac{\alpha-1}{\alpha}},
  \end{equation}
  where $p^\star_t := \max_{x \in \mathcal{X}} \Pr_{X}[\rho(x, X) \le t]$ and $I_\alpha(X; Y)$ is Sibson's $\alpha$-mutual information of order $\alpha>1$ defined as:
  \begin{equation}
    I_{\alpha}(X; Y) := \min_{Q_Y} D_{\alpha}(P_{XY} \| P_X Q_Y).
  \end{equation}
\end{lemma}
We defer the proof to \Cref{sec:missing-proofs5}.

\begin{lemma}\label{lem:sibson-upper}
  In the hidden weighted graph setting, the Sibson's $\alpha$-mutual information satisfies 
  \begin{equation}\label{eq:alpha-information-upper-bound}
    I_{\alpha}(X; Y) \le m \cdot D_\alpha(\mathcal{P} \| \mathcal{Q}).
  \end{equation}
\end{lemma}
We defer the proof to \Cref{sec:missing-proofs5}.

{\noindent\textbf{Proof of the partial recovery lower bound in \Cref{thm:main}}.}
Let $\rho(X_1, X_2) = |X_1 \setminus X_2|$, by \Cref{def:uniformly-sparse}, for any uniformly sparse $\+H$, we have
\begin{equation}
\begin{aligned}
p^\star_{t} = \max_{H_1 \in \+H} \sum_{\ell=m-t}^{m} \Pr_{H_2 \sim \+H} [|H_1 \cap H_2| = \ell]
    &\le \sum_{\ell\ge m-t} \exp(-\ell \log(p_{1M}^{-1}) + o(\log M))\\
    &\le \frac{\exp(-(m-t) \log(p_{1M}^{-1}) + o(\log M))}{1-\exp(-\log(p_{1M}^{-1}))}\\
    &= \exp\left( - (m-t) \log(p_{1M}^{-1}) + o(\log M) \right).
\end{aligned}
\end{equation}
Applying \Cref{lem:sibson-fano} and \Cref{lem:sibson-upper}, the probability of recovering a $\lambda$-fraction of edges satisfies:
\begin{equation}
  \Pr[\rho(X, \hat{X}) \le (1-\lambda) m] \le \exp\left(\frac{\alpha-1}{\alpha}\left(-\lambda m \log(p_{1M}^{-1}) + o(\log M)  + m D_\alpha(\+P \| \+Q)\right)\right).
\end{equation}
If \Cref{asm:renyi-stable-lower} holds, when $D_{\mathrm{KL}}(\+P \| \+Q) \le (\lambda-\eta) \log(p_{1M}^{-1})$ for some constant $\eta$, there exists a constant $c > 0$ such that $\alpha = 1 + \frac{c}{\log(p_{1M}^{-1})}$ satisfies $\alpha\le \beta_n$, where $\beta_n$ is the order from \Cref{asm:renyi-stable-lower} with $\theta=\frac{\eta/2}{\lambda-\eta}$. By monotonicity of \Renyi divergence in its order,
\begin{equation}
  D_{\alpha}(\+P \| \+Q) \le \tp{1+\frac{\eta/2}{\lambda-\eta}} D_{\mathrm{KL}}(\+P \| \+Q) \le (\lambda-\eta/2) \log(p_{1M}^{-1}).
\end{equation}
\begin{equation}
  \begin{aligned}
  \implies \frac{\alpha-1}{\alpha} \left(-(1+o(1)) \lambda m \log(p_{1M}^{-1}) + m D_\alpha(\+P \| \+Q)\right)
  &= \frac{cm (-\eta/2 + o(1))}{\alpha} \to -\infty,
  \end{aligned}
\end{equation}
which implies $\Pr[\rho(X, \hat{X}) \le (1-\lambda) m] \to 0$ as $n \to \infty$. Therefore, any algorithm designed to recover $\lambda$ fraction of the hidden edges fails with high probability.
\qed

\section{AoN under natural distributions}\label{sec:aon}

In this section, we present several sufficient conditions for AoN. While we analyzed the threshold for the `All' phase in \Cref{sec:almost-upper-bound} and demonstrated its tightness at the exponential scale in \Cref{sec:almost-lower-bound}, in this section we employ two distinct methods to provide sufficient conditions for the `Nothing' phase. The first method generalizes the second moment method of~\cite{mossel2023sharp}. The second one uses the I-MMSE relation for Gaussian to lift our almost exact recovery threshold to an AoN phenomenon.
As a result, we are able to show AoN for several natural distributions, in addition to the \Erdos-\Renyi case studied in~\cite{mossel2023sharp}.

Before proceeding, we recall the definition of the minimum mean square error (MMSE). If we relax the recovery target from $\hat{H} \in \mathcal{H}$ to $\hat{X} \in \mathbb{R}^N$, let $X_e := \mathbb{I}[e \in H^*],\ \forall e \in [N]$ and define the mean square error (MSE) as $\mathrm{MSE}(\hat{X}) := \mathbb{E}[\|\hat{X} - X\|^2]$, then MSE is minimized when $\hat{X}$ is chosen to be the conditional expectation $\mathbb{E}[X \mid Y]$  and the MMSE is thus expressed as
\begin{equation}
  \mmse := \mathbb{E}\left[\|X - \mathbb{E}[X \mid Y]\|^2\right] = \mathbb{E}\left[\|X\|^2 - \langle X,\mathbb{E}[X \mid Y]\rangle\right] = \mathbb{E}\left[\|X\|^2\right] - \mathbb{E}\left[\|\mathbb{E}[X \mid Y]\|^2\right].
\end{equation}

Intuitively, if MMSE is close to $0$, one achieves almost exact recovery, while MMSE close to $m$ means almost no recovery (Nothing). In the subsequent proof, we will use MMSE to demonstrate the hardness of recovery in terms of $D_2$. 

\subsection{AoN under distributions with $D_{2}= (1+ o(1)) D_{\mathrm{KL}}$}\label{sec:12renyi-aon}

We recall the \Erdos-\Renyi setting in~\cite{mossel2023sharp}: the weight of every edge in the observed graph can only be $0$ or $1$, and only graphs where all edge weights are $1$ have an identical non-zero posterior probability. Consequently, their partition function can be expressed as the number of graphs whose edge weights are all $1$.  By generalizing their definition of partition function to real-valued distributions, a similar second moment method remains applicable with respect to the \emph{second-order \Renyi divergence} $D_2$. As discussed in \Cref{col:12renyi}, when $D_2(\+P\|\+Q) = (1 + o(1)) D_{\-{KL}}(\+P\|\+Q)$, AoN occurs and it is visualized in \Cref{fig:tikz-12renyi}.

\begin{figure}[h!]
\centering
\begin{tikzpicture}[scale=1.2]
\definecolor{lightred}{RGB}{255,200,200}
\definecolor{verylightred}{RGB}{255,230,230}
\definecolor{lightgreen}{RGB}{230,245,230}
\definecolor{lightblue}{RGB}{230,230,255}

\pgfmathsetmacro{\X}{6}
\pgfmathsetmacro{\Y}{2.5}
\pgfmathsetmacro{\A}{3.6}
\pgfmathsetmacro{\N}{3.6}

\coordinate (O) at (0,0);
\coordinate (X) at (\X,0);
\coordinate (Y) at (0,\Y);
\coordinate (A) at (\A,0);
\coordinate (N) at (\N,0);

\pgfmathsetmacro{\Factor}{\N/\A}
\coordinate (D) at ($(O)!\Factor!(\A,\Y)$);

\fill[lightred] (O) rectangle ($(N) + (Y)$);
\fill[verylightred] (D) -- ($(A) + (Y)$) -- ($(N) + (Y)$) -- cycle;

\fill[lightgreen] (A) rectangle ($(X) + (Y)$);

\fill[lightblue] (N) -- (D) -- ($(A) + (Y)$) -- (A) -- cycle;

\draw[dashed] (O) -- (D);
\draw (N) -- (D) -- ($(A) + (Y)$) -- ($(X) + (Y)$);

\pgfmathsetmacro{\U}{1.2}
\node at ($(0.5*\N, \U)$) {Nothing};
\node at ($(0.5*\A+0.5*\X,\U)$) {All};
\node at ($(0.5*\N,\Y-0.3)$) {Impossible region};
\node at ($(0.5*\A+0.5*\X,\Y-0.3)$) {Possible region};

\draw[->] (O) -- ($(X) + (0.5, 0)$) node[right] {$D_{\mathrm{KL}}$};
\draw[->] (O) -- ($(Y) + (0, 0.5)$) node[above, align=center] {Target edge recovery fraction};

\draw (\A,0.1) -- (\A,-0.1);
\draw (0.1,\Y) -- (-0.1,\Y);

\node[below] at (N) {$D_{\mathrm{KL}}\approx D_2\approx \log(p_{1M}^{-1})$};

\node[left] at (Y) {1};
\node[left] at (O) {0};
\end{tikzpicture}
\caption{Asymptotic visualization of \Cref{col:12renyi}.}
\label{fig:tikz-12renyi}
\end{figure}
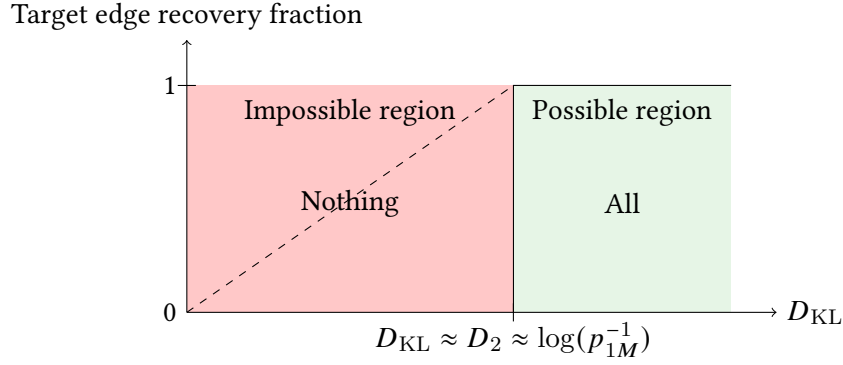

As discussed earlier, the MMSE estimator is the posterior mean $\mathbb E[X\mid Y]$. Here and below we identify a graph with its edge-incidence vector in $\{0,1\}^N$. We will also use a posterior sample $\hat H\sim \Pr[H\mid Y]$ as a convenient device; its conditional mean equals the posterior mean. In the hidden weighted model, the posterior distribution is proportional to the product of the likelihood ratios associated with the edges in the corresponding hidden structure:
\begin{equation}
  \Pr[\hat H \mid Y] = \frac{1}{Z(Y)} \prod_{e \in \hat H} \frac{\-d\+P}{\-d\+Q}(Y_e),
\end{equation}
where $Z(Y)$ is the partition function
\begin{equation}
  Z(Y) := \sum_{H \in \+H} \prod_{e\in H} \frac{\-d\+P}{\-d\+Q}(Y_e).
\end{equation}

\begin{lemma}\label{lem:ZY-upperbound}
Let $M := |\+H|$, for any $\varepsilon > 0$ we have $\Pr[Z(Y) \leq \varepsilon M] \leq \varepsilon$.
\end{lemma}
\begin{proof}
For the purpose of this lemma, we define two models. The planted model, denoted by $\^P$, is defined consistently with the rest of the paper: a graph $H^* \in \mathcal{H}$ is sampled uniformly at random, and subsequently, the observation $Y$ is obtained by independently sampling all edges, such that edges corresponding to $H^*$ are drawn from the distribution $\+P$, and all other edges are drawn from the distribution $\+Q$.
The null model, denoted by $\^Q$, is obtained by independently sampling all edges from the distribution $\+Q$. Then we have
\begin{equation}
    \mathbb{E}_{\^Q}[Z(Y)] = \sum_{H^* \in \+H} \prod_{e\in H^*} \mathbb{E}_{\+Q}\left[\frac{\-d\+P}{\-d\+Q}\right] = M.\\
\end{equation}

And
\begin{equation}
    \frac{\-d\^P}{\-d\^Q}(Y) = \frac{1}{M} \sum_{H^* \in \+H}\prod_{e\in H^*} \frac{\-d\+P}{\-d\+Q}(Y_e) = \frac{Z(Y)}{\mathbb{E}_{Q_Y}[Z(Y)]}.\\
\end{equation}

Therefore
\begin{equation}
\begin{aligned}
\Pr_{\^P}\left[Z(Y) \leq \varepsilon \^E_Q Z(Y)\right] &= \^E_{\^Q} \left[ \^I[Z(Y) \leq \varepsilon \^E_{\^Q} Z(Y)]\cdot \frac{Z(Y)}{\^E_{\^Q} Z(Y)} \right]\\
&\leq \varepsilon\ \Pr_{\^Q}[Z(Y) \leq \varepsilon \^E_{\^Q} Z(Y)] \leq \varepsilon.
\end{aligned}
\end{equation}
\end{proof}

The following lemma can be viewed as a generalization of the second moment approach in Lemma 3.11 of \cite{mossel2023sharp} from Bernoulli inference model to general distributions through $D_2$. 
Similar to contiguity arguments, we need to control a second moment, but we truncate the contribution from overlaps below $o(m)$, since the uniformly sparse condition does not yield a useful upper bound on the overlap probability in this regime.

\begin{lemma}\label{lem:MMSE-lower-bound}
If distributions $\+P, \+Q$ and two graphs $H_1, H_2$ sampled i.i.d. from the graph collection $\mathcal{H}$, satisfy
\begin{equation}\label{eq:MMSE-lower-bound-condition}
\sum_{\ell \geq \delta m} \Pr_{H_1,H_2 \sim \+H}[|H_1 \cap H_2| = \ell] \exp(\ell D_2(\+P\|\+Q)) = o(1),
\end{equation}
then $\mathrm{MMSE}(\+P,\+Q)/m \ge 1-\delta - o(1)$.
\end{lemma}
\begin{proof}
We define the partition function centered at $H^*$ with radius $\ell$ as
\begin{equation}
    Z_\ell(H^*,Y) := \sum_{\hat H \in \+H} \mathbb{I}[|H^*\cap \hat H|=\ell] \prod_{e\in \hat H} \frac{\-d\+P}{\-d\+Q}(Y_e).
\end{equation}

Then we have
\begin{equation}
\begin{aligned}
\mathbb{E} [Z_\ell(H^*,Y)] &= \mathbb{E} \left[ \sum_{\hat H \in \+H} \mathbb{I}[|H^*\cap \hat H|=\ell] \prod_{e\in \hat H} \frac{\-d\+P}{\-d\+Q}(Y_e) \right]\\
&= \frac{1}{M} \sum_{H^*, \hat H \in \+H} \mathbb{I}[|H^*\cap \hat H|=\ell]\mathbb{E}_{Y|H^*} \left[\prod_{e\in \hat H} \frac{\-d\+P}{\-d\+Q}(Y_e) \right]\\
&= \frac{1}{M} \sum_{H^*, \hat H \in \+H} \mathbb{I}[|H^*\cap \hat H|=\ell]
\tp{\prod_{e\in \hat H \cap H^*} \mathbb{E}_{\+P}\left[\frac{\-d\+P}{\-d\+Q}\right]}
\tp{\prod_{e\in \hat H \setminus H^*} \mathbb{E}_{\+Q}\left[\frac{\-d\+P}{\-d\+Q}\right]}\\
&= M \Pr_{H_1,H_2 \sim \+H}[|H_1 \cap H_2| = \ell] \exp(\ell D_2(\+P\|\+Q)).
\end{aligned}
\end{equation}

Applying \Cref{lem:ZY-upperbound}, we have
\begin{equation}
\begin{aligned}
\Pr\left[\frac{|H^* \cap \hat H|}{m} \geq \delta\right]
&\le \Pr[Z(Y) \le \varepsilon M] + \Pr\left[\frac{|H^* \cap \hat H|}{m} \geq \delta , Z(Y) > \varepsilon M\right]\\
&\leq \varepsilon + \frac{1}{\varepsilon M} \mathbb{E}\left[\sum_{\ell \ge \delta m} Z_\ell(H^*, Y)\right]\\
&= \varepsilon + \frac{1}{\varepsilon} \cdot \tp{\sum_{\ell \ge \delta m}\Pr_{H_1,H_2 \sim \+H}[|H_1 \cap H_2| = \ell] \exp(\ell D_2(\+P\|\+Q))}.
\end{aligned}
\end{equation}

Since $\varepsilon$ can be any positive value, when \Cref{eq:MMSE-lower-bound-condition} is satisfied, we have $\Pr\left[\frac{|H^* \cap \hat H|}{m} \geq \delta\right] = o(1)$, and
\begin{equation}
    \mmse/m = \^E\left[\|H^*\|_2^2 - \langle H^*,\^E[\hat H|Y]\rangle\right]/m \ge 1-\delta-o(1).
\end{equation}
\end{proof}

{\noindent\textbf{Proof of the ``Nothing'' recovery lower bound in \Cref{thm:main}}.}

By \Cref{def:uniformly-sparse}, for any constant $\delta \in (0,1)$,
\begin{equation}
\begin{aligned}
    \sum_{\ell \geq \delta m} \Pr_{H_1,H_2 \sim \+H}[|H_1 \cap H_2| = \ell] \exp(\ell D_2(\+P\|\+Q)) &\le \sum_{\ell \geq \delta m} \exp(\ell D_2(\+P\|\+Q) - \ell \log(p_{1M}^{-1}) + o(\log M))\\
    &\le \sum_{\ell \geq \delta m} \exp((-\eta \ell+o(m)) \log(p_{1M}^{-1}))\\
    &\le \frac{\exp((-\eta \delta m + o(m)) \log(p_{1M}^{-1}))}{1 - \exp(-\eta  \log(p_{1M}^{-1}))} = o(1).
\end{aligned}
\end{equation}
Applying \Cref{lem:MMSE-lower-bound}, we have $\^E[\|\^E[X|Y]\|^2] = m - \mmse \le m (\delta + o(1))$. For any estimator $\hat{X}$ satisfying $\|\hat{X}\|^2 = m$, by Cauchy–Schwarz inequality, we have
\begin{equation}
    \mathbb{E}[X\cdot\hat X] = \mathbb{E}[\hat X\cdot \mathbb{E}[X\mid Y]] \le \sqrt{\^E[\|\hat X\|^2]} \cdot \sqrt{\^E[\|\^E[X|Y]\|^2]} = m \sqrt{\delta + o(1)}.
\end{equation}
Thus, the expected overlap of any estimator satisfying $\|\hat X\|^2=m$ is at most $m \sqrt{\delta + o(1)}$. Since this holds asymptotically for any constant $\delta$, Markov's inequality implies that, for every fixed $\lambda,\gamma>0$, choosing $\delta$ sufficiently small gives
\begin{equation}
  \Pr[X\cdot \hat X \ge \lambda m]\le \gamma+o(1).
\end{equation}
Any graph estimator that outputs a member of $\+H$ can be identified with its edge-incidence vector $\hat X\in\{0,1\}^N$; since every graph in $\+H$ has $m$ edges, it satisfies $\|\hat X\|^2=m$, and $X\cdot\hat X=|H^*\cap \hat H|$. Therefore, graph estimators also satisfy the upper bound above.
Hence no algorithm can recover any constant fraction of the edges with any constant probability.

\subsection{AoN under Gaussian distributions}
\label{sec:gaussian-aon}

When $\+P = N(\mu, 1), \+Q = N(0, 1)$, and $\mu > 0$, we observe that $D_{\mathrm{KL}}(\+P \| \+Q) = \frac{1}{2} D_2(\+P \| \+Q)$. In this case, the method presented above does not yield a tight threshold for the `Nothing' phase. Intuitively, this is because there is insufficient concentration when the information approaches the threshold. Nevertheless, we can still prove the existence of AoN for all uniformly sparse $\+H$ using the Gaussian channel's I-MMSE relation.
Our proof has a similar flavor to the \emph{area-theorem} style of establishing AoN in~\cite{niles2021all}, while the I-MMSE relation allows us to establish the ``Nothing'' from ``all'' in a much more straightforward fashion.

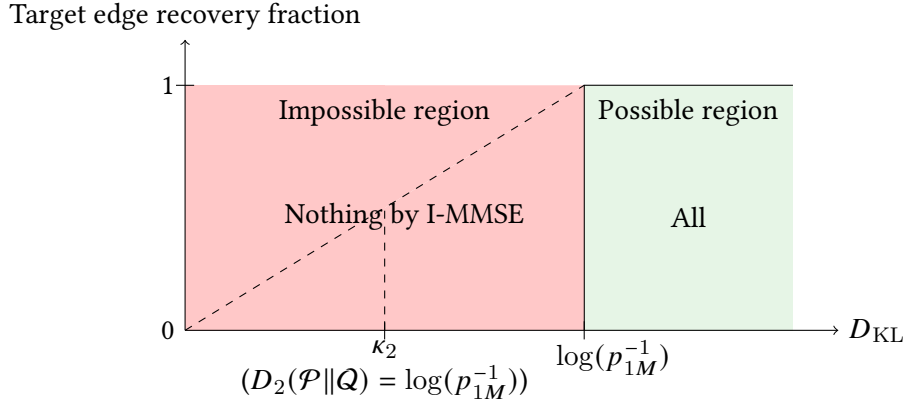
\begin{figure}[h!]
\centering
\begin{tikzpicture}[scale=1.2]
\definecolor{lightred}{RGB}{255,200,200}
\definecolor{verylightred}{RGB}{255,230,230}
\definecolor{lightgreen}{RGB}{230,245,230}
\definecolor{lightblue}{RGB}{230,230,255}

\pgfmathsetmacro{\X}{6.7}
\pgfmathsetmacro{\Y}{2.7}
\pgfmathsetmacro{\A}{4.4}
\pgfmathsetmacro{\N}{2.2}

\coordinate (O) at (0,0);
\coordinate (X) at (\X,0);
\coordinate (Y) at (0,\Y);
\coordinate (A) at (\A,0);
\coordinate (N) at (\N,0);

\pgfmathsetmacro{\Factor}{\N/\A}
\coordinate (D) at ($(O)!\Factor!(\A,\Y)$);

\fill[lightred] (O) rectangle ($(N) + (Y)$);
\fill[lightred] (D) -- ($(A) + (Y)$) -- ($(N) + (Y)$) -- cycle;

\fill[lightgreen] (A) rectangle ($(X) + (Y)$);

\fill[lightred] (N) -- (D) -- ($(A) + (Y)$) -- (A) -- cycle;

\draw[dashed] (O) -- ($(A) + (Y)$);
\draw[dashed] (N) -- (D);
\draw (A) -- ($(A) + (Y)$) -- ($(X) + (Y)$);

\pgfmathsetmacro{\U}{1.25}
\node at ($(0.55*\A, \U)$) {Nothing by I-MMSE};
\node at ($(0.5*\A+0.5*\X,\U)$) {All};
\node at ($(0.5*\A,\Y-0.3)$) {Impossible region};
\node at ($(0.5*\A+0.5*\X,\Y-0.3)$) {Possible region};

\draw[->] (O) -- ($(X) + (0.5, 0)$) node[right] {$D_{\mathrm{KL}}$};
\draw[->] (O) -- ($(Y) + (0, 0.5)$) node[above, align=center] {Target edge recovery fraction};

\draw (\N,0.1) -- (\N,-0.1);
\draw (\A,0.1) -- (\A,-0.1);
\draw (0.1,\Y) -- (-0.1,\Y);

\node[below, align=center] at (N) {$\kappa_2$\\($D_2(\mathcal{P}\|\mathcal{Q}) = \log(p_{1M}^{-1})$)};
\node[below] at (A) {\quad\quad$\log(p_{1M}^{-1})$};

\node[left] at (Y) {1};
\node[left] at (O) {0};
\end{tikzpicture}
\caption{Asymptotic visualization of \Cref{thm:aon}. Here $\kappa_2$ denotes the $D_{\mathrm{KL}}$-coordinate at which $D_2(\mathcal P\|\mathcal Q)=\log(p_{1M}^{-1})$ for Gaussian distributions.}
\label{fig:tikz-gaussian}
\end{figure}

In this setting, let $\snr = \mu^2$ denote the signal-to-noise ratio. The signal can be represented as a vector $X \in \mathbb{R}^N$, where $X_e = \^I[e \in H^*] \ \forall e \in [N]$, and the observation is $Y = \sqrt{\snr} X + Z$, with noise $Z$ being i.i.d. $N(0, 1)$. The I-MMSE relation states that

\begin{lemma}[I-MMSE relation~\cite{guo2005mutual}]
  \label{lem:I-mmse}
  For any random vector $X \in \mathbb{R}^N$, noise $Z$ follows i.i.d. $N(0, 1)$ distribution and $Y = \sqrt{\snr} X + Z$, the mutual information $I(X; Y)$ and the minimum mean squared error satisfy the following relation:
  \begin{equation}
    I(X; Y) = \frac{1}{2} \int_0^{\snr} \mmse(t) \ \mathrm{d}t,
  \end{equation}
	  where $\mmse(t)$ represents the MMSE at signal-to-noise ratio $t$.
\end{lemma}

\begin{lemma}\label{lem:mmse-non-increasing}
  MMSE is non-increasing as a function of the snr.
\end{lemma}
\begin{proof}
  Consider two SNR levels such that $0 < \snr_1 < \snr_2$. We construct the following coupling.
  Let
  \begin{equation}
  Z_2 \sim N(0, \tfrac{1}{\snr_2})^{\otimes N}, \quad
  Z_1 \sim N\left(0, \tfrac{1}{\snr_1} - \tfrac{1}{\snr_2} \right)^{\otimes N},
  \end{equation}
  and define the observations as
  \begin{equation}
  Y_2 = \sqrt{\snr_2} \cdot (X + Z_2), \quad
  Y_1 = \sqrt{\snr_1} \cdot (X + Z_1 + Z_2).
  \end{equation}
  Then $Y_1$ is a noisier version of $Y_2$, obtained via a post-processing step on $Y_2$. By the law of total variance, we have:
  \begin{equation}
    \mmse(\snr_2) 
    = \mathbb{E}[\operatorname{Var}(X \mid Y_2)]
    \overset{(*)}{=} \mathbb{E}[\operatorname{Var}(X \mid Y_1, Z_1)]
    \le \mathbb{E}[\operatorname{Var}(X \mid Y_1)]
    = \mmse(\snr_1),
  \end{equation}
  where $(*)$ holds because $Z_1$ is independent of $Y_2$ and conditioned on $Z_1$, $Y_1$ and $Y_2$ are linearly related. Thus, $\mmse(\snr)$ is non-increasing in $\snr$.
\end{proof}

\begin{lemma}\label{lem:recovery-to-mmse}
 If there exists a graph estimator $\hat H$ (that selects $m$ edges out of the total $N$ edges) that can recover a $\lambda$ fraction of the hidden edges with probability exceeding $\delta$, then it follows that
  \begin{equation}
    \mathrm{MMSE} \le (1 - \delta^2\lambda^2) m.
  \end{equation}
\end{lemma}
\begin{proof}
  Given a graph estimator $\hat H$, for any $\gamma > 0$, we can construct the following vector estimator $\hat X \in \mathbb{R}^N$: let $\hat X_e := \gamma \cdot \mathbb{I}[e \in \hat H] \ \ \forall e \in [N]$. The mean square error of $\hat X$ has the following upper bound:
  \begin{equation}
    \begin{aligned}
      \mathrm{MSE}(\hat X) &= \mathbb{E}\left[\|X - \hat X\|^2\right]\\
      &\le \Pr[|H^* \cap \hat H| \ge \lambda m] \cdot \left( \lambda m \cdot (1-\gamma)^2 + (1-\lambda) m \cdot (1 + \gamma^2) \right)\\
      &\quad + \Pr[|H^* \cap \hat H| < \lambda m] \cdot (1 + \gamma^2) m\\
      &\le \delta \cdot \left( (1 + \gamma^2)m - 2 \gamma \lambda m \right) + (1-\delta) \cdot (1 + \gamma^2) m\\
      &= \left(1 + \gamma^2 - 2 \delta\lambda \gamma\right) m.
    \end{aligned}
  \end{equation}
  Setting $\gamma = \delta \lambda$, we have:
  \begin{equation}
    \mmse \le \mathrm{MSE}(\hat X) \le (1 - \delta^2\lambda^2) m.
  \end{equation}
\end{proof}

\begin{lemma}\label{lem:gaussian-aon}
  For uniformly sparse $\+H$ under Gaussian distributions, if almost exact recovery is achievable when $D_{\mathrm{KL}}(\mathcal{P} \| \mathcal{Q}) = (1 + \eta) \log(p_{1M}^{-1})$ for any constant $\eta > 0$,
  then it follows that no algorithm (regardless of computational efficiency) can recover any constant fraction of the edges with any constant probability when $D_{\mathrm{KL}}(\mathcal{P} \| \mathcal{Q}) = (1 - \eta) \log(p_{1M}^{-1})$ for any constant $\eta > 0$. 
\end{lemma}
\begin{proof}
  For Gaussian distributions, $D_{\mathrm{KL}}(\+P \| \+Q) = \frac{\mu^2}{2} = \frac{\snr}{2}$, we abbreviate $\snr_c := 2 \log(p_{1M}^{-1})$. \;
  
  We proceed by contradiction: suppose there exist constants $\eta_1, \lambda, \delta \in (0,1)$ such that an estimator $\hat H$ can recover a $\lambda$ fraction of the edges with probability exceeding $\delta$ when $D_{\mathrm{KL}}(\mathcal{P} \| \mathcal{Q}) = (1 - \eta_1) \log(p_{1M}^{-1})$ (i.e., when $\mathrm{snr} = (1 - \eta_1) \cdot \mathrm{snr}_c$), then by \Cref{lem:recovery-to-mmse}, we have:
  \begin{equation}
    \mathrm{MMSE}((1 - \eta_1) \cdot \mathrm{snr}_c) \le (1- \delta^2\lambda^2) m.
  \end{equation}
  Set $\eta_2 = \frac{\eta_1 \delta^2\lambda^2}{2 (1-\delta^2\lambda^2)}$, which is a positive constant. Using the I-MMSE relation in \Cref{lem:I-mmse}, along with the monotonicity of the MMSE in \Cref{lem:mmse-non-increasing}, we have the following upper bound for mutual information at $\snr = (1+\eta_2) \snr_c$:
  \begin{equation}
    \begin{aligned}
	      I(X; Y) &= \frac{1}{2} \int_0^{(1+\eta_2) \snr_c} \mmse(t) \ \mathrm{d}t\\
      &\le \frac 12 \left[(1-\eta_1) \snr_c \cdot m + (\eta_1 + \eta_2) \snr_c \cdot (1-\delta^2\lambda^2) m\right]\\
      &= (1 - \frac 12 \eta_1 \delta^2 \lambda^2) \cdot \frac{\snr_c \cdot m}{2} = (1 - \frac 12 \eta_1 \delta^2 \lambda^2) \log M.
    \end{aligned}
  \end{equation}
  However, by the assumption that almost exact recovery is achievable at $D_{\mathrm{KL}}(\mathcal{P} \| \mathcal{Q}) = (1 + \eta_2) \log(p_{1M}^{-1})$ (i.e., at $\mathrm{snr} = (1 + \eta_2) \cdot \mathrm{snr}_c$), \Cref{lem:reverse-fano} gives $I(X; Y) \ge (1 - o(1)) \log M$ at $\mathrm{snr} = (1 + \eta_2) \cdot \mathrm{snr}_c$. This leads to a contradiction.
\end{proof}

Finally, we are ready to complete the proof of AoN under Gaussian distributions.

{\noindent\textbf{Proof of \Cref{thm:aon}}.}
Since the Gaussian distribution satisfies \Cref{asm:renyi-stable-upper}, the corresponding upper bound for recovery has been established in \Cref{thm:main}: almost exact recovery is achievable at $D_{\mathrm{KL}}(\mathcal{P} \| \mathcal{Q}) \ge (1 + \eta) \log(p_{1M}^{-1})$ for any constant $\eta > 0$. Then by \Cref{lem:gaussian-aon}, it follows that no algorithm can recover any constant fraction of the edges with any constant probability when $D_{\mathrm{KL}}(\mathcal{P} \| \mathcal{Q}) = (1 - \eta) \log(p_{1M}^{-1})$ for any constant $\eta > 0$. The same impossibility holds below this value by monotonicity of the Gaussian channel in the signal-to-noise ratio.

\section*{Acknowledgements}
We thank Tselil Schramm for their insightful discussions and helpful comments on the second moment approach.
JL is supported by the National Natural Science Foundation of China under Grant No. 62472212.
\clearpage

\bibliographystyle{alpha}
\bibliography{refs}

\appendix
\crefalias{section}{appendix}
\crefalias{subsection}{appendix}
\crefalias{subsubsection}{appendix}

\section{Missing Proofs}
\label{sec:missing-proofs}

\subsection{Missing proofs of almost exact recovery upper bound}\label{sec:missing-proofs3}

\noindent{\textbf{Proof of \Cref{lem:almost-upper}}.}

Let $L=\frac{\mathrm{d}\+P}{\mathrm{d}\+Q}$. Since $\+P\ll \+Q$, we have $\+P(L=0)=0$, so $X=\log L$ is well-defined under $\+P$. Under $\+Q$, however, $Y=\log L$ may take the value $-\infty$. All log MGFs are understood in the extended-real sense. Define
\begin{equation}\label{eq:log-mgf}
  \begin{aligned}
    \psi_P(\lambda)
    &= \log \mathbb{E}_{\+P} [e^{\lambda X}]
    = \log \int L^\lambda\,\mathrm{d}\+P,\\
    \psi_Q(\lambda)
    &= \log \mathbb{E}_{\+Q} [e^{\lambda Y}]
    = \log \int L^\lambda\,\mathrm{d}\+Q.
  \end{aligned}
\end{equation}

Define the Legendre transforms by
\begin{equation}\label{eq:log-mgf-legendre}
  \begin{aligned}
    E_P(\tau) = \sup_{\lambda\in\mathbb R}\{\lambda\tau-\psi_P(\lambda)\},\qquad
    E_Q(\tau) = \sup_{\lambda\in\mathbb R}\{\lambda\tau-\psi_Q(\lambda)\}.
  \end{aligned}
\end{equation}

Since $\psi_P(0)=\psi_Q(0)=0$, both $E_P(\tau)$ and $E_Q(\tau)$ are nonnegative. Let $D=D_{\mathrm{KL}}(\+P\|\+Q)=\mathbb E_{\+P}[X]$. By Jensen's inequality,
\begin{equation}
  \psi_P(\lambda)
  =
  \log \mathbb E_{\+P}[e^{\lambda X}]
  \ge
  \lambda \mathbb E_{\+P}[X]
  =
  \lambda D.
\end{equation}
Thus $E_P(D)=0$. Moreover, if $0\le\tau\le D$, then for every $\lambda>0$,
\begin{equation}
  \lambda\tau-\psi_P(\lambda)
  \le
  \lambda(\tau-D)
  \le 0.
\end{equation}
Since $\lambda=0$ gives value $0$, positive $\lambda$ cannot increase the supremum. Hence
\begin{equation}
  E_P(\tau)
  =
  \sup_{\lambda\le 0}\{\lambda\tau-\psi_P(\lambda)\},
  \qquad 0\le\tau\le D.
\end{equation}

Similarly, for $E_Q$, take $\lambda<0$. If $\+Q(L=0)>0$, then $\mathbb E_{\+Q}L^\lambda=+\infty$. Otherwise, whenever the expectation is finite, the function $x\mapsto x^\lambda$ is convex on $(0,\infty)$, and Jensen's inequality gives
\begin{equation}
  \mathbb E_{\+Q}L^\lambda
  \ge
  (\mathbb E_{\+Q}L)^\lambda
  =
  1.
\end{equation}
Therefore $\psi_Q(\lambda)\ge0$ for all $\lambda<0$, in the extended-real sense. Hence for every $\tau\ge0$,
\begin{equation}
  \lambda\tau-\psi_Q(\lambda)\le0
  \qquad \text{for all } \lambda<0.
\end{equation}
Again $\lambda=0$ gives value $0$, so negative $\lambda$ cannot increase the supremum. Therefore
\begin{equation}
  E_Q(\tau)
  =
  \sup_{\lambda\ge 0}\{\lambda\tau-\psi_Q(\lambda)\},
  \qquad \tau\ge0.
\end{equation}
Applying Chernoff bound: for $\ell$ i.i.d. random variables $X_1, \dots, X_\ell$ drawn from the distribution $\+X$, and $\ell$ i.i.d. random variables $Y_1, \dots, Y_\ell$ drawn from the distribution $\+Y$, we have:
\begin{equation}\label{eq:xy-chernoff}
  \Pr \left[ \sum_{i=1}^{\ell} X_i \le \ell \tau \right] \le e^{-\ell E_P(\tau)}, \quad \Pr \left[ \sum_{i=1}^{\ell} Y_i \ge \ell \tau \right] \le e^{-\ell E_Q(\tau)}.
\end{equation}

Recall that $\+H(H^*,\ell) := \{H:|H\setminus H^*| = \ell, H \in \+H\}$, define $\+H^\mathrm{in}(H^*,\ell) := \{H^*\setminus H: H \in \+H(H^*,\ell)\}$ and $\+H^\mathrm{out}(H^*,\ell) := \{H\setminus H^*: H \in \+H(H^*,\ell)\}$. For any $\tau \in [0,D_{\mathrm{KL}}(\+P \| \+Q)]$, by union bound we obtain:
\begin{equation}\label{eq:MLE-chernoff-expand}
  \begin{aligned}
    &\Pr[\exists H \in \+H(H^*,\ell): L(H) > L(H^*)]\\
    &\le \^E_{H^* \sim \+H}[|\+H^\mathrm{in}(H^*,\ell)|] \Pr\left[ \sum_{i=1}^{\ell} X_i \le \ell \tau \right] + \^E_{H^* \sim \+H}[|\+H^\mathrm{out}(H^*,\ell)|] \Pr\left[ \sum_{i=1}^{\ell} Y_i \ge \ell \tau \right]\\
    &\le \^E_{H^* \sim \+H}[|\+H^\mathrm{in}(H^*,\ell)|] e^{-\ell E_P(\tau)} + \^E_{H^* \sim \+H}[|\+H^\mathrm{out}(H^*,\ell)|] e^{-\ell E_Q(\tau)}.\\
  \end{aligned}
\end{equation}

When $\+H$ is uniformly sparse, by \Cref{def:uniformly-sparse}, we have
\begin{equation}\label{eq:H-out-bound}
\begin{aligned}
\^E_{H^* \sim \+H}[|\+H^\mathrm{out}(H^*,\ell)|] \le \^E_{H^* \sim \+H}[|\+H(H^*,\ell)|]
&\le M \Pr_{H_1,H_2 \sim \+H}[|H_1 \cap H_2| = m-\ell]\\
&\le \exp(\ell \log(p_{1M}^{-1}) + o(\log M)).
\end{aligned}
\end{equation}

Let $\delta_n \to 0$ to be decided later and for all $\ell \ge \delta_n m$, we have
\begin{equation}\label{eq:H-in-bound}
  |\+H^\mathrm{in}(H^*,\ell)| \le \binom{m}{\ell} \le \left(\frac{e m}{\ell}\right)^\ell \le \left( \frac{e}{\delta_n} \right)^\ell.
\end{equation}

On the interval $[0,D_{\mathrm{KL}}(\+P \| \+Q)]$, the function $E_Q(\tau)$ is nondecreasing in $\tau$, while $E_P(\tau)$ is nonincreasing. We will use the following choice of $\tau$ and lower bound the two exponents separately.

When $D_{\mathrm{KL}}(\+P \| \+Q) \ge (1+\eta) \log(p_{1M}^{-1})$, by \Cref{asm:renyi-stable-upper} there exists a sequence $\epsilon_n \to 0$, such that for $\lambda = -\frac{1}{\epsilon_n \log(p_{1M}^{-1})}$,
\begin{equation}
  D_{1+\lambda}(\+P \| \+Q) \ge \tp{1-\frac{\eta/2}{1+\eta}} D_{\mathrm{KL}}(\+P \| \+Q) \ge (1+\eta/2) \log(p_{1M}^{-1}).
\end{equation}

Setting $\tau = (1 + \eta/4)\log(p_{1M}^{-1})$, then
\begin{equation}\label{eq:Ep-lower-bound}
\begin{aligned}
  E_P(\tau) &\ge \lambda \tau - \psi_P(\lambda)\\
  &= \lambda((1+\eta/4)\log(p_{1M}^{-1}) - D_{1+\lambda}(\+P \| \+Q))\\
  &\ge \lambda (-\eta/4) \log(p_{1M}^{-1}) = \frac{\eta}{4 \epsilon_n},\\
  E_Q(\tau) &\ge  1\cdot \tau - \psi_Q(1) = \tau = (1 + \eta/4) \log(p_{1M}^{-1}).
\end{aligned}
\end{equation}
By combining~\cref{eq:H-in-bound}, ~\cref{eq:H-out-bound} and~\cref{eq:Ep-lower-bound}, and setting $\delta_n = \exp(-\sqrt{1/\epsilon_n})$, then we have,
\begin{equation}\label{eq:MLE-chernoff-final}
  \begin{aligned}
    \^E_{H^* \sim \+H}[|\+H^\mathrm{in}(H^*,\ell)|] e^{-\ell E_P(\tau)} &\le \exp\left( \ell \log\left(\frac{e}{\delta_n}\right) - \frac{\ell \eta}{4 \epsilon_n}\right) \le \exp\left( -(1+o(1))\frac{\ell \eta}{4 \epsilon_n}\right),\\
    \^E_{H^* \sim \+H}[|\+H(H^*,\ell)|] e^{-\ell E_Q(\tau)} &\le \exp\left( \ell \log(p_{1M}^{-1}) + o(\log M) - \ell (1+\eta/4) \log(p_{1M}^{-1}) \right)\\
    &\le \exp\left( -\frac{\ell \eta}{4} \cdot \log(p_{1M}^{-1}) + o(\log M)\right).\\
  \end{aligned}
\end{equation}

Substituting~\cref{eq:MLE-chernoff-final} into~\cref{eq:MLE-chernoff-expand} finishes the proof.
\qed

\subsection{Missing proofs of almost exact recovery lower bound}\label{sec:missing-proofs4}

\begin{lemma}[Distance-based Fano's inequality in \cite{duchi2013distance}]\label{lem:distance-based-fano}
  Let $X$ be uniformly distributed over a finite set $\mathcal{X}$, and let $\rho: \mathcal{X} \times \mathcal{X} \to \mathbb{R}_{\ge 0}$ be a symmetric distance function. For all estimators $\hat{X}$ that form a Markov chain $X \rightarrow Y \rightarrow \hat{X}$, define the error probability $P_t := \Pr[ \rho(\hat{X}, X) > t ]$.
  If $|\mathcal{X}| - N_t^{\min} > N_t^{\max}$, then
  \begin{equation}
    P_t \ge 1 - \frac{I(X; Y) + \log 2}{\log \left( \frac{|\mathcal{X}|}{N_t^{\max}} \right)},
  \end{equation}
	  where $N_t^{\min} := \min_{v \in \mathcal{X}} \left|\{ v' \in \mathcal{X} : \rho(v, v') \le t \} \right|$
	  , $N_t^{\max} := \max_{v \in \mathcal{X}} \left|\{ v' \in \mathcal{X} : \rho(v, v') \le t \} \right|$.
\end{lemma}

{\noindent \textbf{Proof of \Cref{lem:reverse-fano}}.}
  If there exists an algorithm that achieves almost exact recovery, then there exist sequences $\epsilon_n, \delta_n \to 0$ such that $\Pr\left[|H^* \setminus \hat{H}| > \delta_n m\right] \le \epsilon_n$.
  We apply the distance-based Fano's inequality \cite{duchi2013distance} (restated as~\Cref{lem:distance-based-fano}) with distance function $\rho(X_1, X_2) = |X_1 \setminus X_2|$ and $t = \delta_n m$.
  It remains to verify its condition $|\mathcal X|-N_t^{\min}>N_t^{\max}$.
  By \Cref{def:uniformly-sparse},
  \begin{equation}
  \begin{aligned}
    N_t^{\max} &= \max_{H_1 \in \+H} M \sum_{\ell=0}^{\delta_n m} \Pr_{H_2 \sim \+H} [|H_1 \cap H_2| = m-\ell]\\
    &\le \sum_{\ell=0}^{\delta_n m} \exp(\ell \log(p_{1M}^{-1}) + o(\log M))\\
    &\le 2\exp(\delta_n m \log(p_{1M}^{-1}) + o(\log M))
    = \exp(o(\log M)).
  \end{aligned}
  \end{equation}
  Since $N_t^{\min}\le N_t^{\max}$, and $\log M = m\log(p_{1M}^{-1}) \to \infty$, the estimate above implies $N_t^{\max} = M^{o(1)} = o(M)$.
  Hence, for all sufficiently large $n$,
  \begin{equation}
    |\mathcal X|-N_t^{\min} = M-N_t^{\min} \ge M-N_t^{\max} > N_t^{\max},
  \end{equation}
  so \Cref{lem:distance-based-fano} is applicable. Since $P_t \le \epsilon_n$, rearranging gives
  \begin{equation}
  \begin{aligned}
      I(X; Y) &\ge (1-\epsilon_n) \log \left( \frac{|\mathcal{H}|}{N_t^{\max}} \right) - \log 2\\
      &= (1-\epsilon_n) \left(\log M - \log N_t^{\max}\right)-\log 2 = (1 - o(1)) \log M.
    \end{aligned}
  \end{equation}
\qed

{\noindent \textbf{Proof of \Cref{lem:mutual-information-upperbound}}.}
  By definition, for any distribution $Q_Y$ we have
  \begin{equation}
    \begin{aligned}
      I(X;Y) &= \mathbb{E}_{P_{XY}} \left[ \log \frac{\mathrm{d} P_{Y|X}}{\mathrm{d} P_Y} \right] = \mathbb{E}_{P_{XY}} \left[ \log \frac{\mathrm{d} P_{Y|X} \mathrm{d} Q_Y}{\mathrm{d} Q_Y \mathrm{d} P_Y} \right] \\
      &= \mathbb{E}_{P_{XY}} \left[ \log \frac{\mathrm{d} P_{Y|X}}{\mathrm{d} Q_Y} \right] - D_{\mathrm{KL}}(P_Y \| Q_Y) 
      \leq \mathbb{E}_{P_X} \left[ D_{\mathrm{KL}}(P_{Y|X} \| Q_Y) \right].
    \end{aligned}
  \end{equation}
  Let $Q_Y = \+Q^{\otimes N}$, which means the weight of every edge in $Y$ is independently drawn from the distribution $\mathcal{Q}$. For a fixed $X$, the edge weights in the conditional distribution $P_{Y|X}$ are mutually independent. By the additivity of KL divergence under product distributions, we have $D_{\mathrm{KL}}(P_{Y|X} \| Q_Y) = m \cdot D_{\mathrm{KL}}(\mathcal{P} \| \mathcal{Q})$. Then we get \Cref{eq:KL-information-upper-bound}.
\qed

\subsection{Missing proofs of partial recovery lower bound}\label{sec:missing-proofs5}

{\noindent\textbf{Proof of \Cref{lem:sibson-fano}}.}

We build upon the framework established in Theorem 7.2 of \cite{esposito2024sibson}, with a simple modification for the setting of distance-based recovery lower bounds.

Theorem 5.3 of \cite{esposito2024sibson} has shown that
\begin{equation}
  I_{\alpha}(X; Y) = \sup_{f: \mathcal{X} \times \mathcal{Y} \to \mathbb{R}} \frac{\alpha}{\alpha - 1} \log \mathbb{E}_{P_{XY}} \left[ e^{(\alpha - 1) f(X, Y)} \right] - \log \mathbb{E}_{P_X Q_Y^*} \left[ e^{\alpha f(X, Y)} \right],
\end{equation}
where $Q_Y^*$ is the distribution that minimizes the divergence $D_{\alpha}(P_{XY} \| P_X Q_Y)$.

Consider $\alpha > 1, \ \beta > 0$, and let $f(X, \hat{X}) = \frac{\beta}{\alpha - 1} \mathbb{I}[\rho(X, \hat{X}) \le t]$. Define $\hat{p} = \Pr_{P_{X \hat X}}[\rho(X, \hat{X}) \le t]$ and $\hat{q} = \Pr_{P_X Q_{\hat X}^*}[\rho(X, \hat{X}) \le t]$. Since $X$ and $\hat{X}$ are independent in $\hat{q}$, define $p^\star_t := \max_x [\sum_{x': \rho(x,x') \le t} P_X(x')]$, then $\hat{q} \leq p^\star_t$. Therefore,
\begin{equation}
  \begin{aligned}
    I_{\alpha}(X; Y) &\geq I_{\alpha}(X, \hat{X}) \\
    &\geq \frac{\alpha}{\alpha - 1} \log \left( \hat{p} e^{\beta} + 1 - \hat{p} \right) - \log \left( \hat{q} e^{\frac{\alpha}{\alpha - 1} \beta} + 1 - \hat{q} \right) \\
    &\geq \frac{\alpha}{\alpha - 1} \log \left( \hat{p} (e^{\beta} - 1) + 1 \right) - \log \left( p^\star_t e^{\frac{\alpha}{\alpha - 1} \beta} + 1 - p^\star_t \right).
    \end{aligned}
\end{equation}

Setting $\gamma = e^\beta - 1$, and rearranging terms yields
\begin{equation}
  \Pr[\rho(X, \hat{X}) \le t] \leq \frac{\left( p^\star_t (\gamma + 1)^{\frac{\alpha}{\alpha - 1}} + 1 - p^\star_t \right)^{\frac{\alpha - 1}{\alpha}} \exp \left\{ \frac{\alpha - 1}{\alpha} I_\alpha(X, Y) \right\} - 1}{\gamma}.
\end{equation}

Taking $\gamma \to +\infty$, we obtain
\begin{equation}
  \Pr[\rho(X, \hat{X}) \le t] \le \left( p^\star_t e^{I_\alpha(X;Y)} \right)^{\frac{\alpha-1}{\alpha}}.
\end{equation}
\qed

{\noindent \textbf{Proof of \Cref{lem:sibson-upper}}.}
  Let $Q_Y = \+Q^{\otimes N}$, which means the weight of every edge in $Y$ is independently drawn from the distribution $\mathcal{Q}$. Then, we have:
  \begin{equation}
  \begin{aligned}
    I_{\alpha}(X; Y) \le D_{\alpha}(P_{XY} \| P_X Q_Y)
    &= \frac{1}{\alpha-1} \log \mathbb{E}_{P_{XY}} \left[ \left( \frac{\mathrm{d} P_{XY}}{\mathrm{d} P_X Q_Y} \right)^{\alpha - 1} \right]\\
    &= \frac{1}{\alpha-1} \log \mathbb{E}_{P_X} \left[\exp\left( (\alpha-1) D_\alpha(P_{Y|X} \| Q_Y) \right)\right].\\
  \end{aligned}
  \end{equation}
  For a fixed $X$, the edge weights in the conditional distribution $P_{Y|X}$ are mutually independent. By the additivity property of \Renyi divergence under product distributions \cite{van2014renyi}, we have $D_\alpha(P_{Y|X} \| Q_Y) = m \cdot D_\alpha(\mathcal{P} \| \mathcal{Q})$. Substituting into the above equation finishes the proof.
\qed

\section{First moment flat and uniformly sparse graph collections}
\label{sec:uniformly-sparse}

We present two sufficient conditions for uniform sparsity.

\subsection{First moment flat implies uniform sparsity}

\Cref{tbl:p1} presents several graph collections that are first moment flat.

\renewcommand{\arraystretch}{1.2}
\begin{table}[h!]
  \centering
  \begin{tabular}{|c|c|c|c|}
  \hline
  Graph & $M$ & $p_{1M}$ & $p_E$ \\
  \hline
  $k$-clique ($1 \ll k \ll \log n$) & $\binom{n}{k}$ & $(1+o(1)) \tp{\frac{ek}{n}}^{\frac{2}{k-1}}$ & $(1+o(1)) \tp{\frac{ek}{n}}^{\frac{2}{k-1}}$\\
  \hline 
  perfect matching & $(n-1)!! \sim \sqrt{2} \tp{\frac ne}^{n/2}$ & $(1+o(1))\frac en$ & $(1+o(1))\frac en$ \\
  \hline
  Hamiltonian cycle & $\frac 12 (n-1)!$ & $(1+o(1)) \frac en$ & $(1+o(1)) \frac en$ \\
  \hline
  $2k$-NN graph ($k \ll \log n$) & $\frac 12(n-1)!$ & $(1+o(1)) (\frac en)^{\frac 1k}$ & \makecell{$\le n^{-(1+o(1)) \frac {1}{k}}$ \\ See \Cref{sec:pc-2knn}}\\
  \hline
  \end{tabular}
\caption{Graph collections that are first moment flat}
  \label{tbl:p1}
\end{table}

\begin{theorem}\label{lem:large-intersection-probability2}
  Every first moment flat isomorphic graph collection $\+H_H$ satisfying $\log(p_{1M}(\+H_H)^{-1}) = \omega(1)$ is uniformly sparse.
\end{theorem}
We need a few lemmas before proving this theorem.

\begin{lemma}\label{lem:large-intersection-bounds}
For any isomorphic graph collection $\+H_H$, any $H_1 \in \+H_H$, and any $\lambda\in\{1/m,2/m,\ldots,1\}$, we have
\begin{equation}\label{eq:isomorphic-overlap-bound0}
    \Pr_{H_2 \sim \mathcal{H}}[|H_1 \cap H_2| = \lambda m] \le p_E^{\lambda m} \cdot (e/\lambda)^{2\lambda m}.
\end{equation}
\end{lemma}
\begin{proof}
Let $M_{J,H}$ denote the number of subgraphs of $H$ which are isomorphic copies of $J$, and $K_n$ denote the complete graph with $n$ vertices. For any fixed $H_1 \in \mathcal H$, define $\+H^+ := \{H_1 \cap H_2 : |H_1 \cap H_2| = \lambda m, H_2 \in \+H\}$, then
\begin{equation}
  \begin{aligned}
    \Pr_{H_2 \sim \mathcal{H}}[|H_1 \cap H_2| = \lambda m] &\le \Pr_{H_2 \sim \mathcal{H}}[\exists J\in \+H^+,\ J \subseteq H_2] \le \sum_{J \in \+H^+} \Pr_{H_2 \sim \mathcal{H}}[J \subseteq H_2].\\
  \end{aligned}
\end{equation}
As $\+H$ consists of isomorphic graphs, uniformly sampling $H_2$ is equivalent to uniformly sampling a random vertex permutation, thus
\begin{equation}
  \begin{aligned}
    \sum_{J \in \+H^+} \Pr_{H_2 \sim \mathcal{H}}[J \subseteq H_2] &= \sum_{J \in \+H^+} \Pr_{\substack{\sigma \sim \mathrm{permutation(n)}}}[J \subseteq (\sigma \circ H_1)]\\
    &= \sum_{J \in \+H^+} \Pr_{\substack{\sigma \sim \mathrm{permutation(n)}}}[(\sigma^{-1} \circ J) \subseteq H_1]\\
    &= \sum_{J \in \+H^+} \Pr_{J' \sim \+H_{J}}[J' \subseteq H_1]\\
    &= \sum_{J \in \+H^+} \frac{M_{J,H_1}}{M_{J,K_n}} \le \frac{|\+H^+| \cdot (\max_{J \in \+H^+} M_{J,H_1})}{\min_{J \in \+H^+}M_{J,K_n}}.\\
  \end{aligned}
\end{equation}

By the definition of expectation threshold $p_E$, we have
\begin{equation}
    (\min_{J \in \+H^+}M_{J,K_n})^{-1} = (\max_{J \in \+H^+}p_{1M}(\+H_{J}))^{\lambda m} \le (\max_{J \subseteq H} p_{1M}(\+H_{J}))^{\lambda m} = (p_E(\+H_H))^{\lambda m}.
\end{equation}

Therefore
\begin{equation}
  \Pr_{H_2 \sim \mathcal{H}}[|H_1 \cap H_2| = \lambda m] \le p_E^{\lambda m} \cdot |\+H^+| \cdot (\max_{J \in \+H^+} M_{J,H_1}). 
\end{equation}

Since $|\+H^+| \le \binom{m}{\lambda m}$ and $M_{J,H_1} \le \binom{m}{\lambda m}$, we have
\begin{equation}
  \Pr_{H_2 \sim \mathcal{H}}[|H_1 \cap H_2| = \lambda m] \le p_E^{\lambda m} \cdot (e/\lambda)^{2\lambda m}.
\end{equation}
\end{proof}

\noindent{\textbf{Proof of \Cref{lem:large-intersection-probability2}}.}
The case $\ell=0$ is trivial, since the probability is at most $1$. For $\ell\ge1$, set $\lambda=\ell/m$. For first moment flat $\+H$, we have $\log(p_{E}^{-1}) = (1+o(1)) \log(p_{1M}^{-1})$ and $\log(p_{1M}^{-1})=\omega(1)$.
By \Cref{lem:large-intersection-bounds}, 
\begin{equation}
  \begin{aligned}
  \Pr_{H_2 \sim \mathcal{H}}[|H_1 \cap H_2| = \lambda m] &\le \exp\left(- \lambda m \log(p_E^{-1}) + 2 \lambda m \log (e/\lambda)\right)\\
  &\le \exp\left(- (1+o(1))\lambda m \log(p_{1M}^{-1}) + 2 \lambda m + 2 m\right)\\
  &= \exp\left(- \lambda m \log(p_{1M}^{-1}) + o(\log M)\right).\\
  \end{aligned}
\end{equation}
Therefore $\+H$ is uniformly sparse.
\qed

\subsection{Sufficiently large graph collections are uniformly sparse}
We also consider non-isomorphic graph collections $\+H$ that are \emph{sufficiently large}: these are 
$\+H$ that satisfy $\log (p_{1M}^{-1}) = (1 + o(1)) \log (N/m) = \omega(1)$. \Cref{tbl:p2} presents several graph collections that are sufficiently large (the asterisks indicate that they are also first moment flat).

\renewcommand{\arraystretch}{1.2}
\begin{table}[h!]
  \centering
  \begin{tabular}{|c|c|c|c|}
  \hline
  Graph & $M$ & $\log M$ & $\log (p_{1M}^{-1})$\\
  \hline
  Uniform $m$-subgraphs & $\binom{N}{m}$ & $(1+o(1)) m \log(N/m)$ & $(1+o(1)) \log(N/m)$\\
  \hline 
  Trees & $n^{n-2}$ & $(n-2) \log n$ & $(1+o(1)) \log n$\\
  \hline
	  \makecell{$k$-factors \\ ($k \le n^{o(1)}$)} & \makecell{$\exp(\frac{nk}{2} \log \frac{n}{k} + O(nk))$ \\ \cite{mckay1991asymptotic}} & $\frac{nk}{2} \log \frac{n}{k} + O(nk)$ & $(1+o(1)) \log \frac{n}{k}$\\
  \hline
  perfect matching\textsuperscript{*} & $(n-1)!!$ & $(1+o(1)) \frac{n}{2} \log n$ & $(1+o(1)) \log n$\\
  \hline
  Hamiltonian cycle\textsuperscript{*} & $\frac 12 (n-1)!$ & $(1+o(1)) n \log n$ & $(1+o(1)) \log n$\\
  \hline
  \end{tabular}
  \caption{Graph collections that are sufficiently large}
  \label{tbl:p2}
\end{table}

\begin{theorem}\label{lem:large-intersection-probability}
  Every sufficiently large $\+H$ is also uniformly sparse.
\end{theorem}

For sufficiently large $\+H$ and distributions satisfying \Cref{asm:renyi-stable-upper}, we also show in \Cref{sec:topm} that a simple top-$m$ selection based on the likelihood ratio recovers almost all hidden edges with high probability, if arbitrary $m$-edge subsets are allowed.

\begin{lemma}\label{lem:large-intersection-bounds2}
For any graph collection $\+H$, any $H_1 \in \+H$, and any $\lambda\in[0,1]$ with $\lambda m\in\mathbb Z$, we have
\begin{equation}
    \Pr_{H_2 \sim \mathcal{H}}[|H_1 \cap H_2| = \lambda m] \le \frac{1}{M} \exp \left( (1-\lambda)m \log (N/m) + (1+\log 2)m\right).
\end{equation}
\end{lemma}
\begin{proof}
	  For any fixed $H_1 \in \mathcal H$,
	  the cases $\lambda=0$ and $\lambda=1$ are immediate from $\binom{N-m}{m}\le (eN/m)^m$ and $\Pr[H_2=H_1]=1/M$, respectively, so assume $0<\lambda<1$. Then
	  \begin{equation}
    \begin{aligned}
      \Pr_{H_2 \sim \mathcal{H}}[|H_1 \cap H_2| = \lambda m] 
      &\le \frac{1}{M}\sum_{H_2\in \binom{[N]}{m}} \mathbb I[|H_1 \cap H_2| = \lambda m]\\
      &= \frac{1}{M} \binom{m}{\lambda m} \binom{N-m}{(1-\lambda)m}\\
      &\le \frac{1}{M} \left( \frac{e}{\lambda} \right)^{\lambda m} \left( \frac{eN}{(1-\lambda) m} \right)^{(1-\lambda) m}\\
      &= \frac{1}{M} \exp \left( (1-\lambda)m \log (N/m) + m \left[ 1 - \lambda \log \lambda - (1-\lambda) \log (1-\lambda) \right]\right)\\
      &\le \frac{1}{M} \exp \left( (1-\lambda)m \log (N/m) + (1+\log 2)m\right).
    \end{aligned}
  \end{equation}
  The last inequality holds because the binary entropy $-\lambda \log \lambda -(1-\lambda) \log (1-\lambda)$ is bounded by $\log 2$.
\end{proof}

\noindent{\textbf{Proof of \Cref{lem:large-intersection-probability}}.}
When $\log (p_{1M}^{-1}) = (1+o(1)) \log (N/m)$, since $\log M = m \log (p_{1M}^{-1})$, by \Cref{lem:large-intersection-bounds2}, we have
\begin{equation}
\begin{aligned}
  \Pr_{H_2 \sim \mathcal{H}}[|H_1 \cap H_2| = \lambda m] 
  &\le \frac{1}{M} \exp \left( (1-\lambda)m \log (N/m) + (1+\log 2)m\right)\\
  &\le \exp \left( -m \log(p_{1M}^{-1}) + (1-\lambda) m \log (N/m) + (1+\log 2)m \right)\\
  &\le \exp \left( -\lambda m \log(p_{1M}^{-1}) + (1-\lambda)m(\log (N/m)- \log(p_{1M}^{-1})) + (1+\log 2)m \right)\\
  &\le \exp \left( -\lambda m \log(p_{1M}^{-1}) + o(\log M) \right).\\
\end{aligned}
\end{equation}
Therefore $\+H$ is uniformly sparse.
\qed

\section{\Renyi Divergence Stability for common distributions}
\label{sec:renyi-stability}

We briefly recall the definition of \Renyi divergence. For $\alpha \in (0,1) \cup (1,+\infty)$, the \Renyi divergence is defined as
\begin{equation}
\Dalpha := \frac{1}{\alpha - 1} \log \int \left(\frac{\mathrm{d} \+P}{\mathrm{d} \+Q}\right)^{\alpha} \mathrm{d} \+Q.
\end{equation}
When $\alpha = 1$, the \Renyi divergence coincides with the KL divergence:
\begin{equation}
D_1(\+P \| \+Q) := \lim_{\alpha \to 1} \Dalpha = \Dkl.
\end{equation}
Under this definition, \Renyi divergence is non-decreasing in $\alpha$, and continuous in $\alpha$ on the domain
$[0, 1] \cup \left\{ \alpha \in (1, \infty] \mid \Dalpha < \infty \right\}$ (see \cite{van2014renyi}).

Let $\psi_Q$ be the log moment generating function (log MGF) of the random variable $\log \frac{\mathrm{d} \+P}{\mathrm{d} \+Q}$ under measure $\+Q$, defined in~\cref{eq:log-mgf} then:
\begin{equation}
  \Dalpha = \frac{1}{\alpha - 1} \psi_Q(\alpha).
\end{equation}
Using a Taylor expansion around $\alpha = 1$, we obtain
\begin{equation}
  \psi_Q(\alpha) = \psi_Q(1) + \psi'_Q(1)(\alpha - 1) + \frac{1}{2} \psi''_Q(1)(\alpha - 1)^2 + o((\alpha - 1)^2),
\end{equation}
where $\psi_Q(1) = 0$, $\psi'_Q(1) = \Dkl$, and $\psi''_Q(1) = \operatorname{Var}_{\+P}[\log \frac{\mathrm{d} \+P}{\mathrm{d} \+Q}]$. This yields the approximation
\begin{equation}
\Dalpha = \Dkl + \frac{1}{2}(\alpha - 1) \operatorname{Var}_{\+P}\left[\log \frac{\mathrm{d} \+P}{\mathrm{d} \+Q}\right] + o(\alpha - 1).
\end{equation}
However, in our setting, each $n$ corresponds to a pair of distributions $(\+P_n, \+Q_n)$, so a non-asymptotic bound is required. Intuitively, the Taylor expansion suggests that \Cref{asm:renyi-stable-upper} is consistent with $\operatorname{Var}_{\+P}[\log \frac{\mathrm{d} \+P}{\mathrm{d} \+Q}] = o(\Dkl^2)$, while \Cref{asm:renyi-stable-lower} is consistent with $\operatorname{Var}_{\+P}[\log \frac{\mathrm{d} \+P}{\mathrm{d} \+Q}] = O(\Dkl^2)$.

\subsection{Bernoulli distributions}

\renewcommand{\Dkl}{D_{\mathrm{KL}}\left( \mathrm{Bern}(p) \| \mathrm{Bern}(q) \right)}
\renewcommand{\Dalpha}{D_\alpha\left( \mathrm{Bern}(p) \| \mathrm{Bern}(q) \right)}

By definition, the KL divergence and \Renyi divergence between two Bernoulli distributions are given respectively by:
\begin{equation}
  \begin{aligned}
    \Dkl &= p \log \frac{p}{q} + (1-p) \log \frac{1-p}{1-q},\\
    \Dalpha &=\frac{1}{\alpha-1} \log \left[ p^\alpha q^{1-\alpha}+(1-p)^\alpha (1-q)^{1-\alpha} \right].
  \end{aligned}
\end{equation}

\subsubsection{\texorpdfstring{$\mathcal{P} = \mathrm{Bern}(1),\ \mathcal{Q} = \mathrm{Bern}(q)$}{P=Bern(1),Q=Bern(q)}}

In this case,
\begin{equation}
  D_{\mathrm{KL}}(\mathrm{Bern}(1)\|\mathrm{Bern}(q))
  =
  D_{\alpha}(\mathrm{Bern}(1)\|\mathrm{Bern}(q))
  =
  -\log q.
\end{equation}
It is evident that \Cref{asm:renyi-stable-upper} and \Cref{asm:renyi-stable-lower} hold.

\subsubsection{\texorpdfstring{$\mathcal{P} = \mathrm{Bern}(p),\ \mathcal{Q} = \mathrm{Bern}(q)$}{p=Bern(p),q=Bern(q)} with fixed \texorpdfstring{$p \in (0,1)$}{p in (0,1)} and \texorpdfstring{$\Dkl\to\infty$}{Dkl -> infinity}}\label{sec:renyi-stability-bern-hard}

In this case we can only verify that \Cref{asm:renyi-stable-lower} holds:

\begin{lemma}\label{lem:bern-stability}
The \Renyi divergence between Bernoulli distributions has the following bound:
\begin{equation}
  \Bigl| \Dalpha - \Dkl \Bigr| \le \frac{1}{2} \max \left\{ \left| \log \frac{p}{q} \right|, \left| \log \frac{1-p}{1-q} \right| \right\}^2 \cdot |\alpha - 1|.
\end{equation}
\end{lemma}
\begin{proof}
Let $S(\alpha) = p^{\alpha} q^{1-\alpha} + (1-p)^{\alpha} (1-q)^{1-\alpha}$, then
\begin{equation}
  \begin{aligned}
    S'(\alpha) &= p^{\alpha} q^{1-\alpha} \left[\log \frac{p}{q}\right] + (1-p)^{\alpha} (1-q)^{1-\alpha} \left[\log \frac{1-p}{1-q}\right],\\
    S''(\alpha) &= p^{\alpha} q^{1-\alpha} \left[\log \frac{p}{q}\right]^2 + (1-p)^{\alpha} (1-q)^{1-\alpha} \left[\log \frac{1-p}{1-q}\right]^2.
  \end{aligned}
\end{equation}

Let $R(\alpha) = \log S(\alpha)$, then
\begin{equation}\label{eq:bern-r}
  \begin{aligned}
  R''(\alpha) &= \frac{S''(\alpha) S(\alpha) - (S'(\alpha))^2}{S(\alpha)^2}\\
  &= \frac{p^{\alpha} q^{1-\alpha} (1-p)^{\alpha} (1-q)^{1-\alpha} \left[ \log\left(\frac{p}{q}\right) - \log\left(\frac{1-p}{1-q}\right) \right]^2}{\left( p^{\alpha} q^{1-\alpha} + (1-p)^{\alpha} (1-q)^{1-\alpha} \right)^2}\\
  &\le \frac 14 \left[ \log \left(\frac{p}{q}\right) - \log\left(\frac{1-p}{1-q}\right) \right]^2\\
  &\le \max \left\{ \left| \log \frac{p}{q} \right|, \left| \log \frac{1-p}{1-q} \right| \right\}^2.
  \end{aligned}
\end{equation}
Where the first inequality holds because $4xy \leq (x+y)^2$, and the second inequality holds because $|x-y| \leq |x| + |y| \leq 2\max\{|x|, |y|\}$. Performing a Taylor expansion of $R$ at $\alpha = 1$, we have
\begin{equation}
  R(\alpha) = R(1) + R'(1)(\alpha - 1) + \frac{1}{2} R''(\xi) (\alpha - 1)^2, \quad \text{for some}\ \xi \in [\min(\alpha, 1), \max(\alpha, 1)].
\end{equation}
Since $R(1) = \log S(1) = 0,\ R'(1) = \frac{S'(1)}{S(1)} = \Dkl$,
\begin{equation}
  \Dalpha = \frac{R(\alpha)}{\alpha-1} = \Dkl + \frac{1}{2} R''(\xi) (\alpha-1).
\end{equation}
Substituting with~\cref{eq:bern-r} completes the proof.
\end{proof}

\begin{lemma}
  When $p \in (0,1)$ is fixed and $\Dkl\to\infty$, \Cref{asm:renyi-stable-lower} holds.
\end{lemma}
\begin{proof}
  Under these assumptions, 
  \begin{equation}
    \max \left\{ \left| \log \frac{p}{q} \right|, \left| \log \frac{1-p}{1-q} \right| \right\} = O(\Dkl),
  \end{equation}
  and there exists a constant $c > 0$ such that $\max \left\{ \left| \log \frac{p}{q} \right|, \left| \log \frac{1-p}{1-q} \right| \right\} \le c \cdot \Dkl$.

   By \Cref{lem:bern-stability}, for any constant $\theta > 0$, let $\alpha_n = 1 + \frac{2\theta}{c^2 \Dkl}$, then
  \begin{equation}
    \begin{aligned}
      \Dalpha &\le \Dkl + \frac{1}{2} \max \left\{ \left| \log \frac{p}{q} \right|, \left| \log \frac{1-p}{1-q} \right| \right\}^2 \cdot (\alpha-1)\\
      &\le (1 + \theta) \cdot \Dkl.
    \end{aligned}
  \end{equation}
\end{proof}

\subsection{Gaussian distributions}

\renewcommand{\Dkl}{D_{\mathrm{KL}}\left( N(\mu,1) \| N(0,1) \right)}
\renewcommand{\Dalpha}{D_\alpha\left( N(\mu,1) \| N(0,1) \right)}

By definition, the KL divergence and \Renyi divergence between two Gaussian distributions are given respectively by:
\begin{equation}
  \begin{aligned}
    \Dkl &= \frac{\mu^2}{2},\\
    \Dalpha &= \frac{\alpha \mu^2}{2}.
  \end{aligned}
\end{equation}

For \Cref{asm:renyi-stable-upper}, take $\alpha_n = 1 - \Dkl^{-1/2}$; for \Cref{asm:renyi-stable-lower}, take $\beta_n = 1 + \Dkl^{-1/2}$. In both cases, the order parameter tends to $1$ and its distance from $1$ is $\omega(1/\Dkl)$. Moreover, for any $\alpha \in \{\alpha_n,\beta_n\}$,
\begin{equation}
  \begin{aligned}
    \Bigl|\Dalpha - \Dkl\Bigr| &= \frac{\mu^2}{2} |\alpha-1|\\
    &= \Dkl^{1/2}\\
    &= o(\Dkl).
  \end{aligned}
\end{equation}
We have $\Dalpha = (1+o(1)) \Dkl$, which satisfies \Cref{asm:renyi-stable-upper} and \Cref{asm:renyi-stable-lower}.

\subsection{Exponential distributions}

\renewcommand{\Dkl}{D_{\mathrm{KL}}\left( \mathrm{Exp}(\lambda_1) \| \mathrm{Exp}(\lambda_2) \right)}
\renewcommand{\Dalpha}{D_\alpha\left( \mathrm{Exp}(\lambda_1) \| \mathrm{Exp}(\lambda_2) \right)}

By definition, the KL divergence and \Renyi divergence between two Exponential distributions are given respectively by:
\begin{equation}
\begin{aligned}
  \Dkl &= \frac{\lambda_2}{\lambda_1} + \log\left(\frac{\lambda_1}{\lambda_2}\right) - 1,\\
  \Dalpha &= \frac{1}{\alpha - 1} \left[ \alpha \log\left( \frac{\lambda_1}{\lambda_2} \right) - \log\left( \alpha \frac{\lambda_1}{\lambda_2} + 1 - \alpha \right) \right].\\
\end{aligned}
\end{equation}
Since the exponential distribution is invariant under scaling up to a change in the rate parameter, without loss of generality, we can set $t = \lambda_1 / \lambda_2$, then
\begin{equation}
  \begin{aligned}
    \Dkl &= \frac{1}{t} + \log t - 1,\\
    \Dalpha &= \frac{1}{\alpha - 1} \left[ \alpha \log (t) - \log\left( \alpha t + 1 - \alpha \right) \right].\\
  \end{aligned}
\end{equation}

\begin{lemma}\label{lem:exp-stability}
  When $\alpha \in [1 - \delta, 1 + \delta]$ and $t-\delta |t-1|>0$, the \Renyi divergence between Exponential distributions has the following bound:
  \begin{equation}
    \Bigl| \Dalpha - \Dkl \Bigr| \le \frac{(t - 1)^2}{2\left( t - \delta |t - 1| \right)^2} \cdot |\alpha - 1|.
  \end{equation}
\end{lemma}
\begin{proof}
  Let $h(\alpha) = \alpha \log t - \log \left( \alpha t + 1 - \alpha \right)$, then
  \begin{equation}
    \begin{aligned}
      h'(\alpha) &= \log t - \frac{t - 1}{(t-1) \alpha + 1},\\
      h''(\alpha) &= \frac{(t - 1)^2}{\left((t-1) \alpha + 1\right)^2}.
    \end{aligned}
  \end{equation}
  Performing a Taylor expansion of $h$ at $\alpha = 1$, we have
  \begin{equation}
    h(\alpha) = h(1) + h'(1)(\alpha-1) + \frac{1}{2}h''(\xi)(\alpha - 1)^2 \quad \text{for some}\ \xi \in [\min(\alpha, 1), \max(\alpha, 1)],
  \end{equation}
  Since $h(1) = 0,\ h'(1) = \log t - \frac{t - 1}{t} = \Dkl$,
  \begin{equation}\label{eq:exp-dist-taylor}
    \Dalpha = \frac{h(\alpha)}{\alpha-1} = \Dkl + \frac{1}{2} h''(\xi) (\alpha - 1).
  \end{equation}

  For $\xi \in [1 - \delta, 1 + \delta]$, we have the lower bound on the denominator:
  \begin{equation}
  \left| (t - 1) \xi + 1 \right| \geq \left| t - \delta |t - 1| \right|.
  \end{equation}
  Therefore, the derivative is upper bounded by:
  \begin{equation}
  \left| h''(\xi) \right| \leq \frac{(t - 1)^2}{\left( t - \delta |t - 1| \right)^2}.
  \end{equation}
  Substituting this into \cref{eq:exp-dist-taylor} finishes the proof.
\end{proof}

\subsubsection{\texorpdfstring{$\+P = \mathrm{Exp}(\lambda_1),\ \+Q = \mathrm{Exp}(\lambda_2)$}{P=Exp(lambda1),Q=Exp(lambda2)} with \texorpdfstring{$\lambda_1 \gg \lambda_2$}{lambda1 >> lambda2}}

  When $t \gg 1$, $\Dkl = (1+o(1)) \log t$, we will show that \Cref{asm:renyi-stable-upper} and \Cref{asm:renyi-stable-lower} holds:
  
	  For \Cref{asm:renyi-stable-upper}, take $\alpha_n = 1-\Dkl^{-1/2}$; for \Cref{asm:renyi-stable-lower}, take $\beta_n = 1+\Dkl^{-1/2}$. In both cases, the order parameter tends to $1$ and its distance from $1$ is $\omega(1/\Dkl)$. Moreover, for any $\alpha \in \{\alpha_n,\beta_n\}$, we have $\left| h''(\xi) \right| = \frac{(1+o(1)) t^2}{(1+o(1)) t^2} = 1+o(1)$. Thus, there exists a constant $c > 0$ such that $\left| h''(\xi) \right| \leq c$, and
  \begin{equation}
    \begin{aligned}
      \Bigl|\Dalpha - \Dkl\Bigr| &\le c \cdot |\alpha - 1| = o(1).\\
    \end{aligned}
  \end{equation}
  Therefore we have $\Dalpha = (1+o(1)) \Dkl$, which satisfies \Cref{asm:renyi-stable-upper} and \Cref{asm:renyi-stable-lower}.

\subsubsection{\texorpdfstring{$\+P = \mathrm{Exp}(\lambda_1),\ \+Q = \mathrm{Exp}(\lambda_2)$}{P=Exp(lambda1),Q=Exp(lambda2)} with \texorpdfstring{$\lambda_1 \ll \lambda_2$}{lambda1 << lambda2}}

  When $t \ll 1$, $\Dkl = (1+o(1)) \frac{1}{t}$, we will show that \Cref{asm:renyi-stable-lower} holds.

  Let $\alpha_n = 1 + \frac{c}{\Dkl}$, where $c \in (0,\frac 12)$ is a constant, then
  \begin{equation}
    \begin{aligned}
      \frac 12|h''(\xi)| \le \frac{(t - 1)^2}{2\left( t - \delta |t - 1| \right)^2} &\le \frac{1}{2\left( t - \delta |t - 1| \right)^2} = \frac{1}{2\left( (1-c+o(1)) t \right)^2}\\
      &\le (1+o(1)) \frac{1}{t^2} = (1+o(1)) \left[ \Dkl \right]^2.
    \end{aligned}
  \end{equation}
  Therefore, there exists a constant $d > 0$ such that
  \begin{equation}
    \begin{aligned}
      \Dalpha &\le \Dkl + d \Dkl^2 \cdot (\alpha-1).\\
    \end{aligned}
  \end{equation}
	  For any constant $\theta > 0$, we can set $c = \min\{\theta/d,1/4\}$, then
  \begin{equation}
    \Dalpha \le (1+\theta) \Dkl.
  \end{equation}

\subsection{Distributions with $D_{2}= (1+ o(1)) D_{\mathrm{KL}}$}\label{sec:12renyi-check}
We show that for $\+P=\-{Bern}(1-o(1)),\ \+Q=\-{Bern}(o(1))$, and $\+P=\-{Exp}(\lambda_1),\ \+Q=\-{Exp}(\lambda_2)$ with $\lambda_1 \gg \lambda_2$, we have $D_{2}= (1+ o(1)) D_{\mathrm{KL}}$.

\subsubsection{$\+P=\-{Bern}(1-o(1)),\ \+Q=\-{Bern}(o(1))$}
Intuitively, this can be seen as an extension to the planted \Erdos-\Renyi random graph model. Roughly speaking, in a standard planted model, the planted graph will always be a subgraph (corresponding to $\+P=\-{Bern}(1)$), while in our setting, we have to also consider the planted subgraph may undergo a random edge removal process.
So an AoN in this case can be seen as an extension to~\cite{mossel2023sharp}.

We verify $D_{2}= (1+ o(1)) D_{\mathrm{KL}}$ next.

\begin{equation}
\begin{aligned}
D_{\mathrm{KL}}(\mathrm{Bern}(p) || \mathrm{Bern}(q)) &= p \ln \left( \frac{p}{q} \right) + (1-p) \ln \left( \frac{1-p}{1-q} \right) = - (1+o(1)) \log q,\\
D_2(\mathrm{Bern}(p) || \mathrm{Bern}(q)) &= \ln \left( \frac{p^2}{q} + \frac{(1-p)^2}{1-q} \right) = - (1+o(1)) \log q.
\end{aligned}
\end{equation}

\subsubsection{$\+P=\-{Exp}(\lambda_1),\ \+Q=\-{Exp}(\lambda_2)$ with $\lambda_1 \gg \lambda_2$}
This setting was studied previously for almost exact recovery of planted matchings (\cite{moharrami2021planted,ding2023planted}) and spanning trees and Hamiltonian paths/cycles (\cite{moharrami2025planted}).

We verify $D_{2}= (1+ o(1)) D_{\mathrm{KL}}$ next.
\begin{equation}
\begin{aligned}
D_{\mathrm{KL}}(\mathrm{Exp}(\lambda_1) \| \mathrm{Exp}(\lambda_2)) &= \ln \left( \frac{\lambda_1}{\lambda_2} \right) + \frac{\lambda_2}{\lambda_1} - 1 = \ln \left( \frac{\lambda_1}{\lambda_2} \right)-1+o(1),\\
D_2(\mathrm{Exp}(\lambda_1) \| \mathrm{Exp}(\lambda_2)) &= \ln \left( \frac{\lambda_1^2}{2\lambda_1\lambda_2 - \lambda_2^2} \right) = \ln \left( \frac{\lambda_1}{\lambda_2} \right) - \ln(2) + o(1).
\end{aligned}
\end{equation}
Therefore, the Bernoulli regime above and the exponential regime with $\lambda_1 \gg \lambda_2$ both satisfy $D_{2}= (1+ o(1)) D_{\mathrm{KL}}$.

\subsection{\texorpdfstring{Assumption 2 of \cite{ding2020consistent}}{[DWXY20]} implies our \texorpdfstring{\Cref{asm:renyi-stable-upper}}{assumption}}

\renewcommand{\Dkl}{D_{\mathrm{KL}}\left( \+P \| \+Q \right)}
\renewcommand{\Dalpha}{D_\alpha\left( \+P \| \+Q \right)}

In Assumption 2 of \cite{ding2020consistent}, the authors stated the following assumption (with notation slightly modified to avoid conflicts):
\begin{assumption}[Assumption 2 in \cite{ding2020consistent}]
  There exists an absolute constant $c > 0$ such that for all $\theta \in [0,1]$,
  \begin{equation}
    E_P\left((1 - \theta) D(P_n | Q_n)\right) \geq c \theta^2 D(P_n | Q_n),
  \end{equation}
\end{assumption}
where $E_P$ is the Legendre transform defined in~\cref{eq:log-mgf-legendre}, and $D(P_n|Q_n)$ denotes $D_{\mathrm{KL}}(P_n\|Q_n)$ in the notation of~\cite{ding2020consistent}. For $\tau=(1-\theta)\Dkl \in [0,\Dkl]$, convexity gives $E_P(\tau)=\sup_{\lambda\le0}\{\lambda\tau-\psi_P(\lambda)\}$, and Jensen's inequality rules out $\lambda<-1$ from improving over the value at $\lambda=0$. Hence $E_P(\tau)=\sup_{\lambda\in[-1,0]}\{\lambda\tau-\psi_P(\lambda)\}$. Therefore, by this assumption, after reducing $c$ by a constant factor if necessary, for any $\theta \in [0,1]$ there exists a $\lambda \in [-1,0)$ satisfying
\begin{equation}
  \lambda \left((1 - \theta) \Dkl - D_{1+\lambda}(\+P \| \+Q)\right) \geq c \theta^2 \Dkl.
\end{equation}
By the monotonicity of \Renyi divergence in its order and $\lambda<0$, we have $\Dkl \ge D_{1+\lambda}(P_n \| Q_n) > (1 - \theta) \Dkl$. Moreover,
\begin{equation}
  -\lambda \ge \frac{c \theta^2 \Dkl}{\theta \Dkl} = c \theta.
\end{equation}
For any fixed $\theta_0>0$, set $\theta_n=\Dkl^{-1/2}$ and let $\lambda_n<0$ be the order shift obtained above with $\theta=\theta_n$. Let $\alpha_n=1-(c/2)\theta_n$. Then $\alpha_n\to 1$, $1-\alpha_n=\omega(1/\Dkl)$, and $\alpha_n\ge 1+\lambda_n$ for all sufficiently large $n$. By monotonicity of \Renyi divergence,
\begin{equation}
  D_{\alpha_n}(\+P\|\+Q)\ge D_{1+\lambda_n}(\+P\|\+Q)\ge (1-\theta_n)\Dkl\ge (1-\theta_0)\Dkl
\end{equation}
for all sufficiently large $n$. This implies that \Cref{asm:renyi-stable-upper} holds.

\section{Almost all hidden edges by top-\texorpdfstring{$m$}{m} selection for sufficiently large \texorpdfstring{$\+H$}{H}}
\label{sec:topm}

\begin{lemma}\label{lem:topm}
  Suppose \Cref{asm:renyi-stable-upper} holds. When $\log (p_{1M}^{-1}) = (1 + o(1)) \log (N/m) = \omega(1)$ and $D_{\mathrm{KL}}(\+P \| \+Q) = (1+\eta) \log (p_{1M}^{-1})$ for some constant $\eta > 0$, selecting the top $m$ edges with the highest likelihood ratios will, with high probability, output an $m$-edge subset with more than $(1 - o(1))m$ edges belonging to the hidden graph $H^*$.
\end{lemma}
\begin{proof}
Recall that in \Cref{lem:almost-upper}, let $\+X$ denote the random variable $\log \frac{\mathrm{d}\+P}{\mathrm{d}\+Q}$ under the distribution $\+P$, and $\+Y$ denote the same quantity under the distribution $\+Q$. Let $X_1, \ldots, X_m$ be i.i.d. samples from $\+X$, and $Y_1, \ldots, Y_{N-m}$ be i.i.d. samples from $\+Y$. For any threshold $\tau$, define $X = \sum_{i=1}^m \mathbb{I}[X_i \le \tau]$ and $Y = \sum_{j=1}^{N-m} \mathbb{I}[Y_j \ge \tau]$. If there exists a threshold $\tau$ such that, with high probability, $X = Y = o(m)$, then it implies that among the $N$ values, the top $m$ largest values contain at least $(1 - o(1))m$ samples from the $\+P$ distribution, i.e., from the hidden graph $H^*$.  

By \Cref{eq:xy-chernoff} we have $\Pr \left[ X_i \le \tau \right] \le e^{-E_P(\tau)}, \quad \Pr \left[ Y_i \ge \tau \right] \le e^{-E_Q(\tau)}$, which implies $\mathbb{E}[X] \le m e^{-E_P(\tau)}$, $\mathbb{E}[Y] \le (N - m) e^{-E_Q(\tau)}$.

By \Cref{eq:Ep-lower-bound}, when $D_{\mathrm{KL}}(\+P \| \+Q) = (1+\eta) \log (p_{1M}^{-1})$ there exists a $\tau$ such that $E_P(\tau) \ge r_n$ for some $r_n \to \infty$ and $E_Q(\tau) \ge (1 + \eta/4) \log(p_{1M}^{-1})$.
Moreover, when $\log(p_{1M}^{-1}) = (1+o(1)) \log (N/m)$, we have 
\begin{equation}
  \begin{aligned}
    \mathbb{E}[X] &\le m e^{-r_n},\\
    \mathbb{E}[Y] &\le (N - m) e^{-(1+\eta/4) \log(p_{1M}^{-1})} \le (N-m) \cdot (m/N)^{1+\eta/4+o(1)}.
  \end{aligned}
\end{equation}

Since $\mathbb{E}[X] = o(m)$ and $\mathbb{E}[Y] = o(m)$, by Markov's inequality, for any fixed $\gamma > 0$ we have
\begin{equation}
  \begin{aligned}
    \Pr[X \ge \gamma m] &\le \frac{\mathbb{E}[X]}{\gamma m} = o(1),\\
    \Pr[Y \ge \gamma m] &\le \frac{\mathbb{E}[Y]}{\gamma m} = o(1).
  \end{aligned}
\end{equation}

Therefore, with high probability, we have $X = o(m)$ and $Y = o(m)$, which implies that the top $m$ largest values contain at least $(1 - o(1))m$ samples from the $\+P$ distribution, i.e., from the hidden graph $H^*$.
\end{proof}

\section{Tightness of the Partial Recovery Lower Bound}
\label{sec:partial-tightness}

{\noindent\textbf{Proof of \Cref{thm:non-aon}}.}

For any constant $0 < \lambda < 1$, take $\+H$ to be the collection of all $m$-edge subgraphs of $K_n$ with $m=o(N)$ and $m\to\infty$. Let $\+P = \mathrm{Bern}(\lambda),\+Q = \mathrm{Bern}\left(\frac{(1-\lambda) m}{N-m}\right)$. We will demonstrate that this distribution is tight for the partial recovery lower bound in \Cref{sec:partial-lower-bound}.

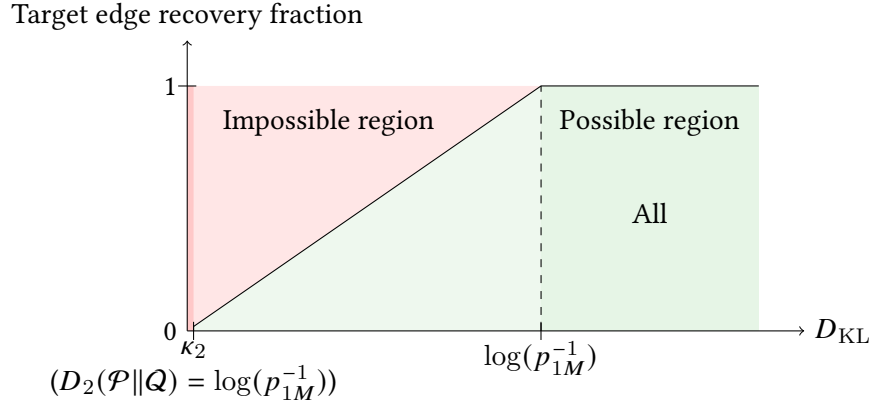
\begin{figure}[h!]
\centering
\begin{tikzpicture}[scale=1.2]
\definecolor{lightred}{RGB}{255,200,200}
\definecolor{verylightred}{RGB}{255,230,230}
\definecolor{lightgreen}{RGB}{230,245,230}
\definecolor{verylightgreen}{RGB}{238,248,238}
\definecolor{lightblue}{RGB}{230,230,255}

\pgfmathsetmacro{\X}{6.3}
\pgfmathsetmacro{\Y}{2.7}
\pgfmathsetmacro{\A}{3.9}
\pgfmathsetmacro{\N}{0.07}

\coordinate (O) at (0,0);
\coordinate (X) at (\X,0);
\coordinate (Y) at (0,\Y);
\coordinate (A) at (\A,0);
\coordinate (N) at (\N,0);

\pgfmathsetmacro{\Factor}{\N/\A}
\coordinate (D) at ($(O)!\Factor!(\A,\Y)$);

\fill[lightred] (O) rectangle ($(N) + (Y)$);
\fill[verylightred] (D) -- ($(A) + (Y)$) -- ($(N) + (Y)$) -- cycle;

\fill[lightgreen] (A) rectangle ($(X) + (Y)$);

\fill[verylightgreen] (N) -- (D) -- ($(A) + (Y)$) -- (A) -- cycle;

\draw[dashed] (A) -- ($(A) + (Y)$);
\draw (N) -- (D) -- ($(A) + (Y)$) -- ($(X) + (Y)$);

\pgfmathsetmacro{\U}{1.3}
\node at ($(0.5*\A+0.5*\X,\U)$) {All};
\node at ($(0.4*\A,\Y-0.4)$) {Impossible region};
\node at ($(0.5*\A+0.5*\X,\Y-0.4)$) {Possible region};

\draw[->] (O) -- ($(X) + (0.5, 0)$) node[right] {$D_{\mathrm{KL}}$};
\draw[->] (O) -- ($(Y) + (0, 0.5)$) node[above, align=center] {Target edge recovery fraction};

\draw (\N,0.1) -- (\N,-0.1);
\draw (\A,0.1) -- (\A,-0.1);
\draw (0.1,\Y) -- (-0.1,\Y);

\node[below, align=center] at (N) {$\kappa_2$\\($D_2(\mathcal{P}\|\mathcal{Q}) = \log(p_{1M}^{-1})$)};
\node[below] at (A) {$\log(p_{1M}^{-1})$};

\node[left] at (Y) {1};
\node[left] at (O) {0};
\end{tikzpicture}
\caption{Asymptotic visualization of \Cref{thm:non-aon}. Here $\kappa_2$ denotes the $D_{\mathrm{KL}}$-coordinate at which $D_2(\+P\|\+Q)=\log(p_{1M}^{-1})$; in this construction, $\kappa_2$ is close to the origin.}
\label{fig:tikz-noaon}
\end{figure}

First we calculate the KL divergence:
\begin{equation}
  D_{\mathrm{KL}}(\+P \| \+Q) = \lambda \log \frac{\lambda}{\frac{(1-\lambda) m}{N-m}} + (1-\lambda) \log \frac{1-\lambda}{1 - \frac{(1-\lambda) m}{N-m}} = (1+o(1)) \lambda \log (N/m).
\end{equation}
This distribution also satisfies \Cref{asm:renyi-stable-lower}, which is proven in \Cref{sec:renyi-stability-bern-hard}.

Next, we demonstrate that it is possible to recover a $\lambda + o(1)$ fraction of edges in the hidden $m$-subgraph model, which satisfies the uniformly sparse condition and it is easier to analyze because under this model, any subgraph with $m$ edges is a valid hidden subgraph.

In the $m$-subgraph model, the inference task can be described as follows:
let $X_1, \dots, X_m$ be i.i.d. samples drawn from distribution $\+P$, and $Y_1, \dots, Y_{N - m}$ be i.i.d. samples drawn from distribution $\+Q$. Given the mixture of these $N$ samples (with their indices randomly permuted), the goal is to select a subset of $m$ indices such that at least a $\lambda + o(1)$ fraction of the selected elements originate from the $\+P$-distributed set.

Define $X := \sum_{i=1}^m X_i$, $Y := \sum_{j=1}^{N-m} Y_j$. Then $\mathbb{E}[X] = \lambda m$, $\mathbb{E}[Y] = (1 - \lambda)m$. We consider the following recovery strategy:
\begin{itemize}
  \item If the total number of one-indices (i.e., $X + Y$) is less than $m$, we select all available one-indices and fill the remaining $m - (X + Y)$ positions arbitrarily with zero-indices. In this case, the number of correctly recovered $\+P$-samples is at least $X$.

  \item If $X + Y \geq m$, we select any $m$ one-indices. Since at most $Y$ one-indices come from $\mathcal{Q}$, the number of correctly recovered indices is at least $m - Y$.
\end{itemize}

Hence, in either case, the number of correctly recovered indices is at least $\min(X, m - Y)$. Using Chernoff bound, we have
\begin{equation}
  \begin{aligned}
    \Pr[|X - \mathbb{E}[X]| \ge \delta \mathbb{E}[X]] &\le 2 \exp\left(-\frac{\delta^2}{3} \cdot \mathbb{E}[X]\right),\\
    \Pr[|Y - \mathbb{E}[Y]| \ge \delta \mathbb{E}[Y]] &\le 2 \exp\left(-\frac{\delta^2}{3} \cdot \mathbb{E}[Y]\right).
  \end{aligned}
\end{equation}
Setting $\delta = m^{-1/3}$, we have
\begin{equation}
  \Pr[\min(X, m - Y) \le \lambda m - m^{2/3}] \le 4 \exp\left(-\frac{\min(\lambda, 1-\lambda)}{3} \cdot m^{1/3}\right).
\end{equation}
This implies that we can recover a $(1 + o(1))\lambda$ fraction of the hidden $m$-subgraph with high probability. Moreover, \Cref{sec:partial-lower-bound} shows that no fixed fraction strictly larger than $\lambda$ can be recovered with high probability under this scaling.

\section{Expectation threshold for \texorpdfstring{$2k$}{2k}-NN graphs}
\label{sec:pc-2knn}

\cite{ding2020consistent} established the almost exact recovery threshold for the $2k$-nearest-neighbor ($2k$-NN) graph. As noted in \Cref{tbl:p1}, the $2k$-NN graph is first-moment flat, which makes \Cref{thm:main} applicable. Here we formally verify first-moment flatness by proving an upper bound on its expectation threshold.

We provide an upper bound on the stronger notion of the ``modified subgraph expectation threshold''. Let $M_{H',H}$ denote the number of subgraphs of $H$ which are isomorphic copies of $H'$, define 
\begin{equation}
\tilde{p}_E(H) = \min \left\{ p : \mathbb{E} M_{H', G(n, p)} \geq \frac{M_{H', H}}{2} \text{ for all edge induced subgraph } H' \subseteq H \right\}.
\end{equation}

By definition, it is clear that $p_E \le \tilde{p}_E$, in fact, ~\cite{mossel2022second} established that $p_c(H) \leq L \tilde{p}_E(H) \log e(H)$ for a universal constant $L$.

In the remainder of this section, we follow the same convention and use ``subgraph'' to only refer to \emph{edge-induced} subgraphs associated with a subset of edges, and simply write $H'\subseteq H$. In particular, $H'$ does not contain isolated vertices.

\begin{lemma}\label{lem:PC-bound}
For $\Delta = n^{o(1)}$ and a graph $H$ of maximum degree $\Delta$,
we have:
\begin{equation}
\tilde{p}_E(H) \le n^{-(1+o(1))\min_{H' \subseteq H}\frac{v(H')-c(H')}{e(H')}}.
\end{equation}
where $v(H'),e(H'),c(H')$ denote the number of vertices, the number of edges, and the number of connected components in $H'$.
\end{lemma}
\begin{proof}
It suffices to bound $\tilde{p}_E(H)$.
Noting that $\mathbb{E} M_{H', G(n, p)} = M_{H',K_n} \cdot p^{e(H')}$ where $K_n$ is the complete graph, we have
\begin{equation}
\begin{aligned}
\tilde{p}_E(H) &= \min \left\{p: \frac{2M_{H',K_n}\cdot p^{e(H')}}{M_{H',H}} \ge 1 \text{ for all } H' \subseteq H \right\}\\
&= \max_{H' \subseteq H} \tp{\frac{M_{H',H}}{2M_{H',K_n}}}^{\frac{1}{e(H')}}\\
&= \max_{H' \subseteq H} \tp{\frac{\mathrm{Aut}(H') M_{H',H}}{2\cdot(n)_{{v(H')}}}}^{\frac{1}{e(H')}}.
\end{aligned}
\end{equation}

Where the term $\mathrm{Aut}(H')$ represents the number of automorphisms of the graph $H'$. An automorphism of a graph $H'$ is a bijection $\phi: V(H') \to V(H')$ that preserves adjacency, meaning that for any two vertices $u, v \in V(H')$, we have $(u, v) \in E(H')$ if and only if $(\phi(u), \phi(v)) \in E(H')$.

We need to find an upper bound for $\frac{\mathrm{Aut}(H') M_{H',H}}{2n^{{v(H')}}}$. By definition,
\begin{equation}
\begin{aligned}
M_{H',H} &= \frac{1}{\mathrm{Aut}(H')} \sum_{\phi:V(H')\hookrightarrow V(H)} \mathbb I[(\phi(u),\phi(v)) \in H,\quad \forall (u,v) \in H'].
\end{aligned}
\end{equation}

Recall that $c(H')$ is the number of connected components of $H'$. Fix any maximal spanning forest in $H'$, so that each connected component has a canonical spanning tree. We note that the spanning tree for a connected component of $H'$ is  also a spanning tree for the induced subgraph of $H$ on the same set of vertices.
For each spanning tree, the first vertex (by lexicographical order) can be mapped to at most $n$ choices under $\phi$. Then, we traverse the spanning tree in a BFS order. Every other vertex of the spanning tree can have at most $\Delta$ choices, because the maximum degree in $H$ is $\Delta$.
Therefore, we have:
\begin{equation}
\mathrm{Aut}(H') M_{H',H} = \sum_{\phi:V(H')\hookrightarrow V(H)} \mathbb I[(\phi(u),\phi(v)) \in H,\quad \forall (u,v) \in H'] \le n^{c(H')} \Delta^{v(H') - c(H')}.
\end{equation}

Then
\begin{equation}
\begin{aligned}
\tilde{p}_E(H) &= \max_{H' \subseteq H} \tp{\frac{\mathrm{Aut}(H') M_{H',H}}{2 \cdot (n)_{{v(H')}}}}^{\frac{1}{e(H')}}\\
&\le \max_{H' \subseteq H} \tp{\frac{n^{c(H')} \Delta^{v(H')-c(H')}}{2\cdot (n)_{{v(H')}}}}^{\frac{1}{e(H')}}\\
&\le \max_{H' \subseteq H} \tp{\frac{n^{c(H')} \Delta^{v(H')-c(H')} }{(n/e)^{v(H')}}}^{\frac{1}{e(H')}}\\
&= \max_{H' \subseteq H} n^{\frac{c(H') - v(H')}{e(H')}\cdot \tp{1-\frac{\log \Delta}{\log n}} + \frac{v(H')}{e(H')\log n}}\\
&= n^{-(1+o(1))\min_{H' \subseteq H}\frac{v(H')-c(H')}{e(H')}}.
\end{aligned}
\end{equation}
The third line follows from Stirling approximation; and the last line follows from noting that $\frac{\log \Delta}{\log n}=o(1)$, and $v(H') \le 2(v(H')-c(H'))$ for $H'$ without isolated points.
\end{proof}

\begin{lemma}\label{lem:PC-bound-kNN}
  For $2k$-NN graph $H$ (every vertex in the graph has a degree of exactly $2k$, $k$ on the left and $k$ on the right), we have:
  \begin{equation}
  \frac{v(H')-c(H')}{e(H')} \ge \frac{1}{k} - O\tp{\frac 1n} \quad \forall H' \subseteq H.
  \end{equation}
  \end{lemma}
  \begin{proof}
  Let $C_n^k$ denote the $2k$-NN graph on the ring. We first record a simple extremal fact: for every nonempty proper vertex subset $S\subsetneq V(C_n^k)$,
  \begin{equation}\label{eq:knn-proper-subset}
    e_{C_n^k}(S) \le k(|S|-1).
  \end{equation}
  Indeed, color the vertices of $S$ black. Among all subsets of a fixed size $s$, the number of pairs within cyclic distance at most $k$ is maximized when the vertices form a consecutive block on the ring: if there are two black segments separated by a gap, moving an endpoint vertex across an adjacent gap to join another black segment cannot decrease the number of black pairs within distance at most $k$, and iterating gives a consecutive block. For a consecutive block of size $s$, if $s\le k+1$ then
  \begin{equation}
    e_{C_n^k}(S) \le \binom{s}{2} \le k(s-1),
  \end{equation}
  while if $k+1<s\le n-k$, then
  \begin{equation}
    e_{C_n^k}(S) \le ks-\frac{k(k+1)}{2}\le k(s-1).
  \end{equation}
  Finally, if $s>n-k$, write $r=n-s$, where $1\le r<k$. Removing the complement of $S$ deletes at least $2kr-\binom r2$ edges, hence
  \begin{equation}
    e_{C_n^k}(S) \le nk-2kr+\binom r2 = ks-\left(kr-\binom r2\right)\le k(s-1).
  \end{equation}

  Fix an arbitrary $H' \subseteq H$, and abbreviate $v(H')$, $e(H')$, and $c(H')$ as $v$, $e$, and $c$. Let $C_1,\dots,C_c$ be the connected components of $H'$, and let $v_i,e_i$ denote their numbers of vertices and edges.
  If $c=1$ and $v=n$, then $e\le nk$, so
  \begin{equation}
    \frac{v-c}{e} \ge \frac{n-1}{nk} = \frac{1}{k}-O\tp{\frac1n}.
  \end{equation}
  Otherwise, the vertex set of each connected component is a nonempty proper subset of the ring. By~\cref{eq:knn-proper-subset}, $e_i\le k(v_i-1)$ for all $i$. Summing over components gives
  \begin{equation}
    e=\sum_{i=1}^c e_i \le k\sum_{i=1}^c(v_i-1)=k(v-c).
  \end{equation}
  Therefore $(v-c)/e\ge 1/k$ in this case.
\end{proof}

Combining \Cref{lem:PC-bound,lem:PC-bound-kNN} with $\Delta=2k$ gives
\begin{equation}
  \tilde p_E(H) \le n^{-(1+o(1))/k}.
\end{equation}
Since $p_E(H)\le \tilde p_E(H)$, we have
\begin{equation}
  \log(p_E(H)^{-1}) \ge (1+o(1))\frac{\log n}{k}.
\end{equation}
On the other hand, $p_E(H)\ge p_{1M}(H)$ by definition, and for the $2k$-NN graph
$p_{1M}(H)=(1+o(1))(e/n)^{1/k}$. Hence
\begin{equation}
  \log(p_E(H)^{-1})
  =
  (1+o(1))\frac{\log n}{k}
  =
  (1+o(1))\log(p_{1M}(H)^{-1}).
\end{equation}
Therefore, the $2k$-NN graph is first moment flat.

\end{document}